\documentclass[final,11pt, oneside, reqno]{article}   	
\usepackage[utf8]{inputenc}
\usepackage{geometry}                		
\usepackage{graphicx}				
\newcommand{\takeout}[1]{\empty}

\usepackage{amssymb, amsmath, wasysym, stmaryrd, multicol, bbm, mathrsfs}
\usepackage{amsthm,thmtools}
\usepackage[affil-it]{authblk}
\usepackage{tikz-cd}
\usepackage{import}
\usepackage[noadjust]{cite} 
\usepackage{amsthm,thmtools}
\usepackage[affil-it]{authblk}
\usepackage{tikz-cd}
\usepackage{import}

\usepackage{dsfont}

\usepackage{xcolor} 

\newcommand{\ownthmSpaceAbove}{5pt}
\newcommand{\ownthmSpaceBelow}{5pt}
\newcommand{\resetCurThmBraces}{%
\gdef\curThmBraceOpen{(}%
\gdef\curThmBraceClose{)}}
\resetCurThmBraces
\newcommand{\removeThmBraces}{%
\gdef\curThmBraceOpen{}%
\gdef\curThmBraceClose{}}

\declaretheoremstyle[
    spaceabove=\ownthmSpaceAbove,
    spacebelow=\ownthmSpaceBelow,
    headpunct=.,
    postheadspace=.5em,
    notebraces={\curThmBraceOpen}{\curThmBraceClose},
    postheadhook={\resetCurThmBraces},
]{definition}
\declaretheoremstyle[
    style=definition,
    bodyfont=\itshape,
    notebraces={\curThmBraceOpen}{\curThmBraceClose},
    postheadhook={\resetCurThmBraces},
]{theorem}

\usepackage{hyperref}
\hypersetup{hidelinks,final,bookmarks}

\usepackage[notref,notcite]{showkeys}
\usepackage{seqsplit}
\usepackage{xstring}
\newcommand{\defaultshowkeysformat}[1]{%
\StrSubstitute{#1}{ }{\textvisiblespace}[\TEMP]%
\parbox[t]{\marginparwidth}{\raggedright\normalfont\small\ttfamily\(\{\){\color{red!50!black}\expandafter\seqsplit\expandafter{\TEMP}}\(\}\)}%
}

\renewcommand*\showkeyslabelformat[1]{%
\noexpandarg%
\defaultshowkeysformat{#1}%
}

\usepackage[footnote,marginclue,nomargin]{fixme}

\usepackage[inline]{enumitem}
\setlist[enumerate,1]{label=(\arabic*),font=\normalfont,align=left,leftmargin=0pt,labelindent=0pt,listparindent=\parindent,labelwidth=0pt,itemindent=!,topsep=3pt,parsep=0pt,itemsep=3pt,start=1}
\setlist[enumerate,2]{label=(\alph*),font=\normalfont,align=left,leftmargin=0pt,labelindent=0pt,listparindent=\parindent,labelwidth=0pt,itemindent=!,topsep=3pt,parsep=0pt,itemsep=3pt,start=1}
\setlist[itemize]{labelindent=*,leftmargin=*,topsep=5pt,itemsep=3pt}
\setlist[description]{labelindent=*,leftmargin=*,itemindent=-1em}

\numberwithin{equation}{section}

\newcommand{\Pos}{\mathsf{Pos}}
\newcommand{\Posf}{\Pos_\mathsf{f}}
\newcommand{\Posz}{\Pos_{0}}
\newcommand{\FinPos}{\mathsf{Fin}(\Pos)}
\newcommand{\Mndf}{\mathop{\mathsf{FinMnd}}}
\newcommand{\Set}{\mathsf{Set}}
\newcommand{\C}{\mathscr C}
\newcommand{\D}{\mathscr D}
\newcommand{\Tcat}{\mathscr T}
\newcommand{\Vcat}{\mathscr V}
\DeclareMathOperator{\Mod}{\mathop{\mathsf{Mod}}}  
\DeclareMathOperator{\Alg}{\mathop{\mathsf{Alg}}}  
\newcommand{\Sig}{\mathsf{Sig}}
\newcommand{\AlgSigma}{\Alg \Sigma}
\newcommand{\AlgcSigma}{\Alg_\mathsf{c} \Sigma}
\newcommand{\V}{\mathcal V}
\newcommand{\T}{\mathbb T}
\newcommand{\M}{\mathbb M}
\newcommand{\E}{\mathcal E}
\newcommand{\TSigma}{T_\Sigma}
\newcommand{\TcSigma}{T^\mathsf{c}_\Sigma}
\newcommand{\Term}{\mathscr{T}}
\newcommand{\TermG}{\Term(\Gamma)}
\DeclareMathOperator{\Lan}{\mathsf{Lan}}
\newcommand{\vt}{\V_\T}
\newcommand{\tv}{\T_\V}
\newcommand{\colim}{\mathop{\mathrm{colim}}}
\newcommand{\obj}{\mathop{\mathsf{obj}}}

\newcommand{\copower}{\bullet}

\newcommand{\N}{\mathds{N}}
\newcommand{\A}{\mathscr{A}}

\newcommand{\id}{\mathsf{id}}

\newcommand{\subto}{\hookrightarrow}
\newcommand{\monoto}{\rightarrowtail}
\newcommand{\epito}{\twoheadrightarrow}
\newcommand{\xra}[1]{\xrightarrow{~#1~}}

\newcommand{\Conv}{{C_\omega}}
\newcommand{\powd}{{\mathcal{P}^\downarrow_\omega}}

\newcommand{\fpair}[1]{\langle #1 \rangle}

\newcommand{\card}[1]{\mathop{\mathsf{card}} #1}

\newcommand{\cotensor}{\pitchfork}
\def\wh{\widehat}

\newcommand{\op}{\mathsf{op}}

\theoremstyle{theorem}
\newtheorem{theorem}{Theorem}[section]
\newtheorem{proposition}[theorem]{Proposition}
\newtheorem{corollary}[theorem]{Corollary}
\newtheorem{lemma}[theorem]{Lemma}

\theoremstyle{definition}
\newtheorem{definition}[theorem]{Definition}
\newtheorem{example}[theorem]{Example}
\newtheorem{notation}[theorem]{Notation}
\newtheorem{observation}[theorem]{Observation}
\newtheorem{remark}[theorem]{Remark}

\takeout{
\makeatletter
\def\@maketitle{%
  \newpage
  \null
  \vskip 2em%
  \begin{center}%
  \let \footnote \thanks
    {\Large\bfseries \@title \par}%
    \vskip 1.5em%
    {\normalsize
      \lineskip .5em%
      \begin{tabular}[t]{c}%
        \@author
      \end{tabular}\par}%
    \vskip 1em%
    {\normalsize \@date}%
  \end{center}%
  \par
  \vskip 1.5em}
\makeatother}

\begin{document}
\FXRegisterAuthor{ja}{aja}{JA}
\FXRegisterAuthor{cf}{acf}{CF}
\FXRegisterAuthor{sm}{asm}{SM}
\FXRegisterAuthor{ls}{als}{LS}

\title{Finitary Monads on the Category of Posets}

\author{Ji\v{r}\'{i} Ad\'{a}mek%
	\thanks{Supported by the Grant Agency of the Czech
          Republic under the grant 19-0092S.}
      }
\affil{\small{Department of Mathematics, Technical University of Prague,
  Czech Republic, and \\
  Institute of Theoretical Computer Science, Technical University Braunschweig, Germany}}
\author{Chase Ford%
  \thanks{Supported by Deutsche Forschungsgemeinschaft
    (DFG, German Research Foundation) as part of the Research
    and Training Group 2475 ``Cybercrime and Forensic Computing" (grant
    number 393541319/GRK2475/1-2019).}
}
\author{Stefan Milius%
  \thanks{Supported by Deutsche Forschungsgemeinschaft (DFG) under
    projects MI~717/5-2 and MI~717/7-1.}
}

\author{Lutz Schr\"{o}der}
\affil{\small{Department of Computer Science,
    Friedrich-Alexander-Universit\"{a}t Erlangen-N\"{u}rnberg (FAU),
    Germany}
}
\maketitle

\centerline{Dedicated to John Power on the occasion of his 60$^\text{th}$ birthday.}

\begin{abstract}
  Finitary monads on $\Pos$ are characterized as the precisely the
  free-algebra monads of varieties of algebras. These are classes of
  ordered algebras specified by inequations in context.  Analagously,
  finitary enriched monads on $\Pos$ are characterized: here we work
  with varieties of coherent algebras which means that their
  operations are monotone.
\end{abstract}


\section{Introduction}\label{S:intro}

Equational specification usually applies classes of (often many-sorted) 
finitary algebras specified by equations. That is, varieties of algebras 
over the category $\Set^S$ of $S$-sorted sets. This is well known to 
be equivalent to applying finitary monads over $\Set^S$, i.e.~monads 
preserving filtered colimits: every variety $\V$ yields a free-algebra 
monad $\T_{\V}$ on $\Set^S$ which is finitary and whose Eilenberg-Moore 
category is isomorphic to $\V$. Conversely, every finitary monad $\T$ on 
$\Set^S$ defines a canonical $S$-sorted variety $\V$ whose free-algebra 
monad is isomorphic to $\T$. 

There are cases in which algebraic specifications use partially
ordered sets rather than sets without a structure. The goal of our
paper is to present for the category $\Pos$ of partially ordered
sets\smnote{Don't delete; otherwise $\Pos$ is not defined.} an
analogous characterization of finitary monads: we define varieties of
ordered algebras which allow us to represent (a)~all finitary monads
on $\Pos$ and (b)~all enriched finitary monads on $\Pos$ as the free-algebra
monads of varieties. `Enriched' refers to $\Pos$ as a cartesian closed
category: a monad is enriched if its underlying functor $T$ is \emph{locally
monotone} ($f\leq g$ in $\Pos(A, B)$ implies $Tf\leq Tg$ in $\Pos(TA, TB))$.
Case (b)~works with algebras on posets such that the operations are monotone 
(and as morphisms we take monotone homomorphisms). Whereas for (a)~we have 
to work with algebras on posets whose operations are not necessarily monotone 
(but whose morphisms are). To distinguish these cases, we shall call an algebra 
\emph{coherent} if its operations are all monotone. 

A basic step, in which we follow the excellent presentation of
finitary monads on enriched categories due to Kelly and
Power~\cite{KP93}, is to work with operation symbols whose arity is a
finite poset rather than a natural number; we briefly recall the
approach of \emph{op. cit.} in \autoref{S:eqpres}. Just as natural numbers $n=\{0,1,\dots, n-1\}$ 
represent all finite sets up to isomorphism, we choose a representative set 
\[
  \Posf
\]
of finite posets up to isomorphism. Members of $\Posf$ are called \emph{contexts}\lsnote{I suggest calling them \emph{arities}; in fact that term is used in the next sentence.}.
 A \emph{signature} is then a set $\Sigma$ of operation symbols of arities from $\Posf$. 
 More precisely, $\Sigma$ is a collection of sets $(\Sigma_\Gamma)_{\Gamma \in
  \Posf}$. Thus, a $\Sigma$-algebra is a poset $A$ together with an
operation $\sigma_A$, for every $\sigma \in \Sigma_\Gamma$, which
assigns to every monotone map $u\colon \Gamma \to A$ an element
$\sigma_A(u)$ of $A$. 
For example, let $\mathbbm{2}$ be the two-chain in $\Posf$ given by $x<y$. 
Then an operation symbol $\sigma$ of arity $\mathbbm{2}$ is interpreted in an 
algebra $A$ as a partial function $\sigma_A\colon A\times A\rightarrow A$ whose 
definition domain consists of all comparable pairs in $A$.

Given a signature $\Sigma$ we form, for every context
$\Gamma\in\Posf$, the set $\TermG$ of \emph{terms in context}
$\Gamma$. It is defined as usual in universal algebra by ignoring the
order structure of contexts. Then, for every $\Sigma$-algebra $A$,
whenever a monotone function $f\colon\Gamma\rightarrow A$ is given
(i.e.~whenever the variables of context $\Gamma$ are interpreted in
$A$) we define an evaluation of terms in context $\Gamma$. This is a
partial map $f^\#$ assigning a value to a term $t$ provided that
values of the subterms of $t$ are defined and respect the order of
$\Gamma$. This leads to the concept of \emph{inequation in context}
$\Gamma$: it is a pair $(s, t)$ of terms in that context. An algebra
$A$ \emph{satisfies} this inequation if for every monotone
interpretation $f\colon\Gamma\rightarrow A$ we have that both
$f^\#(t)$ and $f^\#(s)$ are defined and $f^{\#}(s)\leq f^{\#}(t)$
holds in $A$. We use the following notation for inequations in
context:
\[
  \Gamma\vdash s\leq t.
\]

By a \emph{variety} we understand a category $\V$ of $\Sigma$-algebras 
presented by a set $\E$ of $\Sigma$-inequations in context. Thus the objects 
of $\V$ are all algebras satisfying each $\Gamma\vdash s\leq t$ in $\E$, and 
morphisms are monotone homomorphisms. We prove that every variety $\V$ 
is monadic over $\Pos$, that is, for the monad $\T_{\V}$ of free $\V$-algebras 
$\V$ is isomorphic to the category $\Pos^{\T_{\V}}$ of algebras for $\T_{\V}$. 
Moreover, $\T_{\V}$ is a finitary monad and, in case $\V$ consists of coherent 
algebras, $\T_{\V}$ is enriched.

Conversely, with every finitary monad $\T$ on $\Pos$ we associate a 
canonical variety whose free-algebra monad is isomorphic to $\T$. This 
process from monads to varieties is inverse to the above assignment 
$\V\mapsto\T_{\V}$. Moreover, if $\T$ is enriched, the canonical variety 
consists of coherent algebras. This leads to a bijection between finitary enriched 
monads and varieties of coherent algebras.

Is it really necessary to work with signatures of operations with
partially ordered arities and terms in context? There is a `natural'
concept of a variety of ordered (coherent) algebras for classical
signatures $\Sigma=(\Sigma_n)_{n\in\N}$. Here terms are elements of
free $\Sigma$-algebras on finite sets (of variables) and a variety is
given by a set of inequations $s\leq t$ where $s$ and $t$ are
terms. Such varieties were studied e.g.~by Bloom and
Wright~\cite{Bloom76, BW83}. Kurz and Velebil~\cite{KV17}
characterized these classical varieties as precisely the exact
categories (in an enriched sense) with a `suitable' generator. In a
recent paper, the first author, Dost\'al, and Velebil~\cite{ADV20} proved that for 
every such variety $\V$ the free-algebra monad
$\T_{\V}$ is enriched and \emph{strongly finitary} in the sense of
Kelly and Lack~\cite{KL93}. This means that the functor $T_{\V}$ is
the left Kan extension\smnote{Please don't delete; I find `reflexive
coinserter' rather unclear, whereas everyone with knowledge of basic
category theory knows Kan extensions.}
of its restriction along the full embedding
$E\colon\Pos_{\mathsf{fd}}\hookrightarrow\Pos$ of finite discrete
posets:
\[
  T_{\V}= \Lan_E (T_{\V}\cdot E).
\]
Conversely, every strongly finitary monad on $\Pos$ is isomorphic to
the free-algebra monad of a variety in this classical sense. This
answers our question above affirmatively: contexts are necessary if
\emph{all} (possibly enriched) finitary monads are to be characterized
via inequations.

\begin{example}
  We have mentioned above a binary operation $\sigma(x, y)$ in context
  $x<y$.\smnote{I think $<$ is correct here.}\cfnote{I agree, $<$ is
    correct.}\lsnote{We should write $\le$ in the contexts, as that is
    what these constraints mean (they are not meant to enforce that
    $x$ and $y$ are different)} For the corresponding variety
  $\AlgSigma$ (with no specified inequations) the free-algebra monad
  is described in \autoref{E:TX}. This monad is not strongly
  finitary~\cite[Ex.~3.15]{ADV20}\lsnote{Shouldn't we prove this claim?}, thus no variety
  with a classical signature has this monad as the free-algebra
  monad.\lsnote{A more natural, if slightly more complicated example
    are `bounded-join-semilattices', i.e.~partial orders in which
    every bounded pair of elements has a binary join.}
\end{example}

\paragraph{Related work}
As we have already mentioned, the idea of using signatures 
in context stems from the work of Kelly and Power~\cite{KP93}.
They presented enriched monads by operations and equations. A
signature in their sense is more general than what we use: it is a
collection of \emph{posets} $(\Sigma_\Gamma)_{\Gamma \in \Posf}$, and
a $\Sigma$-algebra $A$ is then a poset together with a monotone
functions from $\Sigma_\Gamma$ to the poset of monotone functions from
$\Pos(\Gamma, A)$ to $A$ for every context $\Gamma$.

Whereas we deal with the monadic view on varieties of ordered algebras
in the present paper, the view using algebraic theories has been
investigated by Power with coauthors, e.g.~\cite{Pow99, PP01, PP02,
  NP09}, see \autoref{S:enriched}. In particular, the paper~\cite{NP09} works with enriched
categories over a monoidal closed category $\Vcat$\smnote{This should
  be the Vcat macro; the V-macro is for varieties.} for which a
$\Vcat$-enriched base category $\C$ has been chosen. Then enriched
algebraic $\C$-theories are shown to correspond to $\Vcat$-enriched
monads on $\C$. This is particularly relevant for the current paper:
by choosing $\Vcat=\Set$ and $\C=\Pos$ we treat non-enriched finitary
monads on $\Pos$, whereas the choice $\Vcat=\C=\Pos$ covers 
the enriched case.




\paragraph{Acknowledgement} The authors are grateful to Ji\v{r}\'i\
Rosick\'y for fruitful discussion.

\section{Equational Presentations of Monads}\label{S:eqpres}

We now recall the approach to equational presentations of finitary
monads introduced by Kelly and Power~\cite{KP93}; our aim here is to
bring the rest of the paper into this perspective. However, we note
that the signatures used here are more general than those of the
subsequent sections, and (unlike later) some enriched category theory
is used. The reader can decide to skip this section without losing the
connection.

For a locally finitely presentable category $\C$ enriched over a
symmetric monoidal closed category $\Vcat$ Kelly and Power consider
(enriched) monads on $\C$ that are finitary, i.e.~the ordinary
underlying endofunctor preserves filtered colimits. Below we
specialize their approach to $\C = \Pos$ considered as an ordinary
category ($\Vcat = \Set$) or as a category enriched over itself
($\Vcat = \Pos$) as a cartesian closed category. In the first case,
the hom-object $\Pos(A,B)$\smnote{Please do not introduce any bracket
notation for hom-objects here; it is not needed!} is the \emph{set} of all monotone functions
from $A$ to $B$; in the latter case, this is the \emph{poset} of those
functions, ordered pointwise. As in \autoref{S:intro}, a representative
set $\Posf$ of finite posets (called \emph{contexts}) is chosen which
is to be viewed as a full subcategory of $\Pos$. We denote by
\[ 
{\mid}\Posf{\mid}
\]
the corresponding discrete category.

\begin{definition}\label{D:ndsig}
  A \emph{signature} is a functor from ${\mid}\Posf{\mid}$ to $\Pos.$
  In other words, a signature $\Sigma$ is a collection of posets
  $\Sigma_{\Gamma}$ of \emph{operation symbols in context} $\Gamma$
  indexed by $\Gamma\in\Posf$. A morphism $s\colon \Sigma \to \Sigma'$
  of signatures, being a natural transformation, is thus just a family
  of monotone maps $s_\Gamma\colon \Sigma_\Gamma \to \Sigma_\Gamma'$
  indexed by contexts.
  
  We denote by
  \[
    \Sig = [|\Posf|, \Pos]
  \]
  the category of signatures and their morphisms.
\end{definition}

\noindent In the introduction we considered the special case of
signatures where each poset $\Sigma_{\Gamma}$ is discrete, i.e.~we
just have a \emph{set} of operation symbols in context $\Gamma$; for
emphasis, we will call such signatures \emph{discrete}.

\begin{remark}\label{R:tensor}
  Recall~\cite[Def.~6.5.1]{Borceux94-2} the concept of a \emph{tensor} for objects $V \in \Vcat$ and
  $C \in \C$: it is an object $V\otimes C$ of $\C$ together a natural
  isomorphism
  \[
    \C(V \otimes C, X) \cong \Vcat(V, \C(C,X)).
  \]
  in $\Vcat$ which is $\V$-natural in $X$. Here $\Vcat(-,-)$ denotes
  the internal hom-functor of $\Vcat$.

  In the case where $\C = \Pos$ and $\Vcat = \Set$ we get the
  copower
  \[\textstyle
    V\otimes C= \coprod_{V} C,
  \]
  and for $\C=\Vcat=\Pos$ we just get the product in $\Pos$:
  \[
    V\otimes C= V\times C.
  \]
\end{remark}


\begin{notation}
  \begin{enumerate}
  \item We denote by $\FinPos$\smnote{I propose yet another notation
      for the categories of finitary endofunctors and monads.}
    the enriched category of finitary enriched endofunctors on
    $\Pos$. In the case where $\Vcat = \Set$, these are all
    endofunctors preserving filtered colimits. For $\Vcat = \Pos,$
    these are all locally monotone endofunctors preserving filtered
    colimits.
  \item The category of finitary enriched monads on $\Pos$ is denoted
    by $\Mndf(\Pos)$. We have a forgetful functor $U\colon \Mndf(\Pos)
    \to \FinPos$.   
  \end{enumerate}
\end{notation}

\noindent By precomposing endofunctors with the non-full embedding
$J\colon{\mid}\Pos_f{\mid}\rightarrow\Pos$ we obtain a forgetful
functor from $\FinPos$ to $\Sig$. It has a left adjoint\smnote{Kelly
  and Power did not notice that because they never work with
  endofunctors.  I think they go right from signatures to monads.}
assigning to every signature~$\Sigma$ the \emph{polynomial functor}
$P_\Sigma$\smnote{I proposed and Jirka agreed to change $H_\Sigma$ to
  $P_\Sigma$ in this section.} given on objects by
\begin{equation}\label{eq:KPpoly}\textstyle
  P_\Sigma X = \coprod_{\Gamma \in \Posf} \Pos(\Gamma, X) \otimes \Sigma_\Gamma,
\end{equation}
and similarly on morphisms. As previously explained, the hom-object
$\Pos(\Gamma,X)$ can have one of the two meanings: for $\V = \Set$
this is regarded as a set and for $\V = \Pos$ as a poset. Henceforth,
we will use that notation for hom-objects only in the latter case and
write
\[
  \Posz(\Gamma,X)
\]
for the set of monotone maps. 
\begin{observation}\label{O:twoenrichments}
  The usual category of algebras for the functor $P_{\Sigma}$, whose
  objects are posets $A$ with a monotone map
  $\alpha\colon P_{\Sigma}A\to A$, has the following form for our two
  enrichements:
  \begin{enumerate}
  \item Let $\Vcat = \Set$. Then $\alpha$ as above  is a monotone map 
    \[\textstyle
      \coprod_{\Gamma\in \Posf}\coprod_{u\in\Posz(\Gamma, A)}\Sigma_{\Gamma}\to A,
    \]
    and as such has components assigning to every monotone function
    $u\colon\Gamma\rightarrow A$ (that is, a monotone interpretation of the
    variables in~$\Gamma$) a monotone function
    $\Sigma_{\Gamma}\to A$. We denote this function by
    $\sigma\mapsto\sigma_A(u)$.
  
    In other words, the poset $A$ is equipped with operations
    $\sigma_A\colon \Posz(\Gamma, A)\to A$ (which need not be monotone
    since $\Posz(\Gamma, A$) is just a set) satisfying
    $\sigma_A(u) \leq \tau_A(u)$ for all pairs $\sigma \leq \tau$ in
    $\Sigma_\Gamma$ and $u$ in $\Pos(\Gamma, A)$. If $\Sigma$ is
    discrete, this is precisely a
    $\Sigma$-algebra (see the \hyperref[S:intro]{introduction}).
  
  \item Now let $\Vcat=\Pos$. Then $\alpha \colon P_{\Sigma}A\to A$ is a monotone map
    \[\textstyle
      \coprod_{\Gamma \in \Posf}\Pos(\Gamma, A)\times\Sigma_{\Gamma}\to A,
    \]
    and thus has as components monotone functions
    $(u, \sigma)\mapsto\sigma_A(u).$ That is, in addition to the
    condition that $\sigma_A(u) \leq \tau_A(u)$ for all pairs
    $\sigma \leq \tau$ in $\Sigma_\Gamma$ and $u$ in $\Pos(\Gamma, A)$
    as above, we also see that each $\sigma_A$ is monotone. Thus, if
    $\Sigma$ is discrete, this is
    precisely a coherent algebra (again, see the \hyperref[S:intro]{introduction}).
  \end{enumerate}
\end{observation}
\noindent Observe also that `homomorphism' has the usual meaning: a
monotone function preserving the given operations. In fact, given
algebras $\alpha\colon P_\Sigma A\to A$ and
$\beta\colon P_\Sigma B\to B$ a homomorphism is a monotone function
$f\colon A\to B$ such that $f\cdot\alpha=\beta\cdot P_\Sigma f$.\smnote{Let
  us please use cdot for composition everywhere; not juxtaposition,
  which makes things hard to read.}  This is equivalent to
$f(\sigma_A(u))=\sigma_B(f\cdot u)$ for all $u\in\Pos(\Gamma, A)$ and
all $\sigma\in\Sigma_{\Gamma}$.
\begin{remark}
  \begin{enumerate}
  \item As shown by Trnkov\'{a} et al.~\cite{TrnkovaEA75} (see also
    Kelly~\cite{Kelly80})\smnote{I think it's better to keep also the
      citation of Kelly because our paper is for John} every
    ordinary finitary endofunctor $H$ on $\Pos$ generates a free monad
    whose underlying functor $\widehat{H}$ is a colimit of the
    $\omega$-chain
      \[
      \widehat{H}=\mathsf{colim}_{n<\omega} W_n
      \]
      of functors, where
      \[
      W_0=\mathsf{Id}\qquad\text{and}\qquad W_{n+1}=HW_n+\mathsf{Id}
      \]
      Connecting morphisms are
      $w_0\colon\mathsf{Id}\to H+\mathsf{Id},$ the coproduct
      injection, and $w_{n+1}=Hw_n+\mathsf{Id}$. The colimit
      injections $c_n\colon W_n X \to \widehat{H}X$ in $\Pos$ have the
      property that if a parallel pair $u,v\colon \widehat{H}X \to A$
      satisfies $c_n \cdot u \leq c_n \cdot v$ for all $n < \omega$,
      then we have $u \leq v$. It follows that $\hat{H}$ is enriched
      if $H$ is.
      
  \item The category of $H$-algebras is isomorphic to the Eilenberg-Moore category 
    $\Pos^{\widehat{H}}$~\cite{Bar70}.
    
  \item Lack~\cite{Lac99} shows that the forgetful functor
    \[
      \Mndf(\Pos)
      \xra{U} \FinPos
      \xra{J} \Sig
    \]
    is monadic. The corresponding monad $\M$ on $\Sig$ assigns to
    every signature $\Sigma$ the signature $\wh{P_\Sigma} \cdot
    J\colon |\Posf| \to \Pos$.
    
  \item It follows that every enriched finitary monad $\T$ on $\Pos$
    can be regarded as an algebra for the monad $\M$. Therefore,
    $\T$ is a coequalizer in $\Mndf(\Pos)$ of a parallel pair of
    monad morphisms between free $\M$-algebras on
    signatures $\Delta,\Sigma$:
    \[
      \begin{tikzcd}
        \wh{P_\Delta}
        \arrow[shift left]{r}{\ell}
        \arrow[shift right]{r}[swap]{r}
        &
        \wh{P_\Sigma}
        \arrow{r}{c}
        &
        \T.
      \end{tikzcd}
    \]
    This is the equational presentation of $\T$ considered by Kelly
    and Lack~\cite{KL93}. 
  \end{enumerate}
\end{remark}
\begin{example}
  \begin{enumerate}
  \item In the case where $\Vcat = \Set$ and $\C = \Pos$,
    $\Mndf(\Pos)$ is the category of (non-enriched) finitary monads on
    $\Pos$. Consider the above coequalizer in the special case that
    $\Delta$ consists of a single operation $\delta$ of context
    $\Gamma$.  That is, $\Delta_{\Gamma}=\{\delta\}$ and all
    $\Delta_{\bar{\Gamma}}$ for $\bar{\Gamma}\neq\Gamma$ are empty. By
    the Yoneda lemma, $l$ and $r$ simply choose two elements of
    $\widehat{H}_{\Sigma}\Gamma$, say $t_l$ and $t_r$. The above
    coequalizer means that $\T$ is presented by the signature $\Sigma$
    and the equation $t_l=t_r$\lsnote{replaced $x$ with~$t$}.
    \newline\indent For $\Delta$ arbitrary, we do not get one
    equation, but a set of equations (one for every operation symbol
    in $\Delta$) and $\T$ is presented by $\Sigma$ and the
    corresponding set of equations, grouped by their respective
    contexts.
  
  \item
  The case $\Vcat=\C=\Pos$ yields as $\Mndf(\Pos)$ the category of enriched finitary 
  monads on $\Pos$. That is, the underlying endofunctor $T$ is locally monotone.
  \end{enumerate}
\end{example}
\begin{remark}
  The fact that every finitary (possibly enriched) monad on $\Pos$ has
  an \emph{equational} presentation depends heavily on the fact that
  signatures are not reduced to the discrete ones. In contrast, we make
  do with discrete signatures in the rest of the paper, and then
  obtain a characterization of finitary (possibly enriched) monads
  using \emph{inequational} presentations. While it is clear that the
  two specification formats are mutually convertible, inequational
  presentations seem natural for varieties of algebras
  on~$\Pos$.\lsnote{Reworded this to make the distinction clearer
    (also, everyone please lose the phrase `in our paper'); is this
    still diplomatic enough? JA: Yes.}

  Of course, it is possible to translate $\Sigma$-algebras for
  non-discrete signatures $\Sigma$ as varieties of algebras for
  discrete ones (see \autoref{ex:var}\ref{ex:var:7}). Using the result
  of Kelly and Power, such a translation would lead to a
  correspondence between finitary monads and varieties. This paper can
  be viewed as a detailed realization of this.
\end{remark}

\section{Varieties of Ordered Algebras}

Recall that $\Posf$ is a fixed set of finite posets that represent
all finite posets up to isomorphism. If $\Gamma\in\Posf$ has the
underlying set $\{x_0,\dots, x_{n-1}\}$, then we call the~$x_i$ the
\emph{variables} of~$\Gamma$. Recall that all monotone functions from
$A$ to $B$ form a set $\Posz(A,B)$ and a poset $\Pos(A,B)$ with the
pointwise order.

\begin{notation}
  The category $\Pos$ is cartesian closed, with hom-objects
  $\Pos(X, Y)$ given by all monotone functions $X\rightarrow Y$,
  ordered pointwise. That is, given monotone functions
  $f,g\colon X\rightarrow Y$, by $f\leq g$ we mean that $f(x)\leq g(x)$
  for all $x\in X$.

 We denote by $|X|$ the underlying set of a poset $X$. We also often
  consider $|X|$ to be the discrete poset on that set.
\end{notation}

\begin{definition}\label{D:sig}
  A \emph{signature in context}\lsnote{This term is not ideal} is a
  set $\Sigma$ of operation symbols each with a prescribed context,
  its \emph{arity}. That is, $\Sigma$ is a collection
  $(\Sigma_{\Gamma})_{\Gamma\in\Posf}$ of sets $\Sigma_{\Gamma}$.  A
  $\Sigma$-\emph{algebra} is a poset $A$ together with, for every
  $\sigma\in\Sigma_{\Gamma}$, a function
  \[
   \sigma_A\colon \Posz(\Gamma, A)\rightarrow A.
  \]
  That is,~$\sigma_A$ assigns to every monotone valuation
  $f\colon \Gamma\rightarrow A$ of the variables in $\Gamma$ an
  element $\sigma_A(f)$ of $A$. The algebra $A$ is called
  \emph{coherent} if each $\sigma_A$ is monotone, i.e.\ whenever
  $f\leq g$ in $\Pos(\Gamma, A)$, then
  $\sigma_A(f)\leq \sigma_A(g)$.
\end{definition}
\begin{notation}
  We denote by $\AlgSigma$ the category of $\Sigma$-algebras. Its
  morphisms $A\to B$ are the \emph{homomorphisms} in the expected
  sense; i.e.\ they are monotone functions $h\colon A\rightarrow B$
  such that for every context $\Gamma$ and every operation symbol
  $\sigma\in\Sigma_{\Gamma}$, the square 
  \[
    \begin{tikzcd}
      {\Posz(\Gamma, A)}
      \arrow[d, "h\cdot (-)"']
      \arrow[r, "\sigma_A"]
      &
      A \arrow[d, "h"]
      \\
      {\Posz(\Gamma, B)}
      \arrow[r, "\sigma_B"]
      &
      B
    \end{tikzcd}
  \]
  commutes. Similarly, we have the category $\AlgcSigma$ of all
  coherent $\Sigma$-algebras. For their homomorphisms we have the
  commutative squares
  \[
    \begin{tikzcd}
      {\Pos(\Gamma, A)} \
      \arrow{d}[swap]{h\cdot (-)}
      \arrow[r, "\sigma_A"]
      &
      A \arrow[d, "h"]
      \\
      {\Pos(\Gamma, B)}
      \arrow[r, "\sigma_B"]
      &
      B                
    \end{tikzcd}
  \]
\end{notation}
\begin{example}\label{E:lin}
  Let $\Sigma$ be the signature given by
  \[
    \Sigma_{\mathbbm{2}}=\{+\}\quad\text{and}\quad \Sigma_{\mathbbm{1}}=\{@\},
  \]
  where $\mathbbm{2}$ is a $2$-chain and $\mathbbm{1}$ is a
  singleton. A $\Sigma$-algebra consists of a poset $A$ with a (not
  necessarily monotone) unary operation $@_A$ and a partial binary
  operation $+_A$ whose definition domain is formed by all comparable
  pairs. Moreover, $A$ is coherent iff both $@_A$ and $+_A$ are
  monotone, the latter in the sense that $a+a'\leq b+b'$ whenever
  $a\leq a', b\leq b', a\leq b$, and $a'\leq b'$.
\end{example}
\takeout{
\smnote{Since this is a paper devoted to John, I strongly propose
  \emph{not} to put the following discussion in numbered remarks or even
  supress them but rather bring them out nicely.}%
Our notion of signature in context is inspired by Kelly and Power's
notion of a signature in an enriched locally finitely presentable
category~\cite{KP93}. In fact, our notion is a special case of their
notion instantiated in $\Pos$. That instance defines a signature in
$\Pos$ as a functor $\Sigma\colon |\Posf| \to \Pos$, where $|\Posf|$
denotes the discrete category with finite posets as
objects.\smnote{You might think that the notation $|\Posf|$ clashes
  with $|X|$. However, it's the same! Just recall that a poset $X$ is
  a special category, and the underlying set $|X|$ is the discrete
  category with objects the objects of $X$. So I feel free to overload
  this notation.}  More concretely, $\Sigma$ assigns to each finite
poset a \emph{poset} $\Sigma_\Gamma$ in lieu of a just a set as in
\autoref{D:sig}. Hence, our notion is the special case where each
$\Sigma_\Gamma$ is a discrete poset.

Following the lead of Kelly and Power further, we now explain how the
notions of $\Sigma$-algebras and coherent ones, respectively, are
instances of the same concept in the enriched setting. This requires
that we recall the notion of a copower (see e.g.~Kelly's
book~\cite[Sec.~3.7]{Kelly82}).
\begin{remark}\label{R:copower}
  \begin{enumerate}
  \item Let $\Vcat$ be a monoidal closed category, and suppose that
    $\C$ is a category enriched over $\Vcat$. The \emph{copower} of an
    object $X$ of $\C$ by an object $V$ of $\Vcat$ is an object $V
    \copower C$ with natural isomorphisms
    \[
      \C(V \copower X, Y) \cong \Vcat(V,\C(X,Y))
      \qquad\text{for every object $Y$ of $\C$}, 
    \]
    where $\Vcat(-,-)$ denotes the internal hom of $\Vcat$. 
    Copowers are frequently called \emph{tensors} and a
    $\Vcat$-category having all copowers is called
    \emph{tensored}.
    
  \item The notion of copower obviously depends on the
    enrichement. The category $\C = \Pos$ can be regarded as an
    enriched category over $\Vcat = \Set$ or over itself: $\Vcat =
    \Pos$. Therefore, we obtain two possible instances of the notion of a
    copower:
    \begin{enumerate}
    \item For $\Vcat = \Set$, the copower of $X$ by $V$ has a natural
      isomorphism $\Pos(V \copower X, Y) \cong \Set(V,
      \Pos(X,Y))$. This implies that the copower is the coproduct
      \[
        V \copower X = \coprod\nolimits_{v \in V} X,
      \]
      whence it is the usual copower.
      
    \item For $\Vcat = \Pos$, the copower of $X$ by $V$ has the
      natural isomorphisms $\Pos(V \copower X, Y) \cong \Pos(V,
      \Pos(X,Y))$. This implies that the copower is the product
      \[
        V \copower X = V \times X.
      \]
    \end{enumerate}
  \end{enumerate}
\end{remark}

Every signature $\Sigma\colon |\Posf| \to \Pos$ (in the sense of Kelly
and Power) gives rise to a polynomial functor on $\Pos$ canonically
obtained by forming the left Kan-extension of $\Sigma$ along the
embedding $J\colon |\Posf| \subto \Pos$:
\[
  \Lan_J \Sigma.
\]
In this case, the usual coend formula for computing left
Kan-extensions~\cite[Thm.~X.4.1]{MacLane98} simplifies to a coproduct. 
Hence, the polynomial functor associated to $\Sigma$ is defined on
objects by
\[
   X \mapsto \coprod\nolimits_{\Gamma \in \Posf} \Pos(\Gamma, X) \copower \Sigma_\Gamma
\]
and similarly on morphisms. In the above formula $\Pos(\Gamma, X)$ can
mean a set or a poset and $\bullet$ can be one of the two copowers we
mentioned in \autoref{R:copower} depending on which enrichment on
$\Pos$ one would like to consider. Moreover, we shall now explain
that by instantiating the copower with those two variants we obtain
that the algebras for the above functor are precisely all
$\Sigma$-algebras or all coherent ones, respectively.
}
\noindent Similarly to the more general signatures discussed in
\autoref{S:eqpres}, signatures~$\Sigma$ in our present sense can be
represented as polynomial functors~$H_\Sigma$ (for $\Sigma$-algebras)
and $K_\Sigma$ (for coherent $\Sigma$-algebras), respectively,
introduced next. These functors arise by specializing the
corresponding instances of the polynomial functor $P_\Sigma$ according
to \autoref{O:twoenrichments} to discrete signatures.
\begin{notation}
  The \emph{polynomial} and \emph{coherent polynomial} functors for a
  signature $\Sigma$ are the endofunctors
  $H_{\Sigma}\colon \Pos \to \Pos$ and $K_\Sigma\colon \Pos \to \Pos$ given by
  \[
    H_\Sigma X = \coprod\nolimits_{\Gamma \in \Posf} \Sigma_\Gamma \times \Posz(\Gamma, X)
    \qquad\text{and}\qquad
    K_\Sigma X = \coprod\nolimits_{\Gamma \in \Posf} \Sigma_\Gamma  \times \Pos(\Gamma, X),
  \]
  respectively, where we regard the sets $\Sigma_\Gamma$ and
  $\Posz(\Gamma,X)$ as discrete posets.  Thus, the elements of both
  $H_\Sigma X$ and $K_\Sigma X$ are pairs $(\sigma, f)$ where $\sigma$
  is an operation symbol of arity~$\Gamma$ and
  $f\colon \Gamma\rightarrow X$ is monotone. The action on
  monotone maps $h\colon X\rightarrow Y$ is then the same for both
  functors:
  \[
    H_{\Sigma}h(\sigma, f) = (\sigma, h\cdot f) = K_\Sigma h(\sigma,f).
  \]
\end{notation}
\begin{remark}
  \begin{enumerate}
  \item \lsnote{Quite honestly I believe these considerations are
      better discharged by a two-line statement accompanied by the
      admittedly dangerous word `clearly' ;) }
      \cfnote{I almost agree with the above: that these categories are isomorphic \emph{is}
      clear; I don't mind leaving the details on concreteness of the iso though. }
      Every $\Sigma$-algebra
    $A$ induces an $H_\Sigma$-algebra
    $\alpha\colon H_{\Sigma}A\rightarrow A$ given by
    \[
      \alpha(\sigma, f)=\sigma_A(f) 
      \qquad
      \text{for $\sigma \in \Sigma_\Gamma$ and $f \in \Posz(\Gamma,X)$.}
    \]
    Conversely, every $H_{\Sigma}$-algebra
    $\alpha\colon H_{\Sigma}A\rightarrow A$ can be viewed as a
    $\Sigma$-algebra, putting $\sigma_A(f)=\alpha(\sigma, f)$. More
    conceptually, we have bijective correspondences between the
    following (families of) maps:
    \[
      \begin{array}{r@{\,}l@{\quad}l}
        \alpha\colon & H_\Sigma A \to A
        \rule[-6pt]{0pt}{0pt}
        \\
        \hline
        \alpha_\Gamma\colon & 
        \Sigma_\Gamma \times \Posz(\Gamma, A) \to A & (\Gamma \in \Posf)
        \rule[-8pt]{0pt}{22pt}
        \\
        \hline
        \sigma_A\colon & \Posz(\Gamma, A) \to A
        & (\Gamma\in \Posf, \sigma \in \Sigma_\Gamma)
        \rule{0pt}{14pt}
      \end{array}
    \]
    Thus, $\AlgSigma$ is isomorphic to the category $\Alg H_\Sigma$ of
    algebras for $H_{\Sigma}$ whose morphisms from $(A, \alpha)$ to
    $(B, \beta)$ are those monotone maps $h\colon A\rightarrow B$ for
    which the square below commutes:
    \[
      \begin{tikzcd}
        H_{\Sigma}A
        \arrow[d, "H_{\Sigma}h"']
        \arrow[r, "\alpha"]
        &
        A
        \arrow[d, "h"]
        \\
        H_{\Sigma}B
        \arrow[r, "\beta"]
        &
        B                
      \end{tikzcd}
    \]
    Indeed, this is equivalent to $h$ being a homomorphism of
    $\Sigma$-algebras. Shortly,
    \[
      \AlgSigma\cong\Alg H_\Sigma.
    \]
    Moreover, this isomorphism is concrete, i.e.~it preserves the
    underlying posets (and monotone maps). That is, if
    $U\colon\AlgSigma\rightarrow\Pos$ and
    $\bar{U}\colon\Alg H_\Sigma\rightarrow\Pos$ denote the forgetful
    functors, the above isomorphism
    $I\colon \AlgSigma\rightarrow\Alg H_\Sigma$ makes the following
    triangle commutative:
    \[  
      \begin{tikzcd}[column sep = 10]
        \AlgSigma
        \arrow[rr, "I"]
        \arrow[rd, "U"']
        &&
        \Alg H_\Sigma
        \arrow[ld, "\bar{U}"]
        \\
        &
        \Pos                                               
      \end{tikzcd}
    \]
  \item Similarly, every coherent $\Sigma$-algebra defines an algebra
    for $K_\Sigma$, and conversely. Indeed, giving an algebra structure
    $\alpha\colon K_\Sigma A \to A$ is to give a context-indexed
    family of monotone maps
    \[
      \alpha_\Gamma\colon \Sigma_\Gamma \times \Pos(\Gamma, A) \to A.
    \]
    Equivalently, we have for every $\sigma$ of arity $\Gamma$ a
    monotone map $\sigma_A\colon \Pos(\Gamma, A) \to A$.

    This leads to an isomorphism $\AlgcSigma \cong \Alg K_\Sigma$,
    which is concrete:
    \[
      \begin{tikzcd}[column sep = 10]
        \AlgcSigma
        \arrow[rr, "I_c"]
        \arrow[rd, "U_c"']
        &&
        \Alg K_\Sigma
        \arrow[ld, "\bar{U}_c"]
        \\
        &
        \Pos                                               
      \end{tikzcd}
    \]
    where $I_c$, $U_c$ and $\bar U_c$ denote the isomorphism and the
    forgetful functors, respectively.
  \end{enumerate}
\end{remark}
\begin{remark}\label{R:fact}
  \lsnote{Do we still need this given the new proof of reflexivity?}
  Recall that epimorphisms in $\Pos$ are precisely the surjective
  monotone maps.  $\Pos$ has the factorization system\lsnote{We should
    decide at some point how much category theory we want to assume.}
  \[
    (\text{epimorphism}, \text{embedding})
  \]
  where \emph{embeddings} are maps $m\colon A\rightarrow B$ such that
  for all $a, a'\in A$ we have $a\leq a'$ iff $m(a)\leq m(a')$. That
  is, embeddings are order-reflecting monotone
  functions.\lsnote{Should we add a categorical description of what
    embeddings are, such as regular monos?}
    \cfnote{It certainly couldn't hurt to add this in a sentence in my opinion}

Given an $\omega$-chain of embeddings in $\Pos$, its colimit is simply their union (with inclusion maps as the colimit cocone). 
\end{remark}

\begin{proposition}\label{P:free}
  Every poset $X$ generates a free $\Sigma$-algebra $\TSigma X$. Its
  underlying poset is the union of the following $\omega$-chain of
  embeddings in $\Pos$:
  \begin{equation}\label{eq:chain}
    W_0= X
    \xra{w_0}
    W_1 = H_{\Sigma}X + X
    \xra{w_1}
    W_2 = H_{\Sigma}W_1 + X 
    \xra{w_3}
    \cdots
  \end{equation}
  where $w_0$ is the right-hand coproduct injection
  $X \to H_\Sigma X + X$ and
  $w_{n+1} = Hw_n + \id_X \colon W_{n+1} = H_\Sigma W_n + X \to
  HW_{n+1} + X = W_{n+2}$ for every $n$. The universal map
  $\eta_X\colon X \to \TSigma X$ is the inclusion of~$W_0$ into the
  union.
\end{proposition}
\begin{proof}
  Observe first that the polynomial functor $H_\Sigma$ can be
  rewritten, up to natural isomorphism, as
  \[
    H_\Sigma X \cong \coprod\nolimits_{\Gamma \in \Posf}
    \coprod\nolimits_{\Sigma_\Gamma} \Posz(\Gamma, X),
  \]
  because every $\Sigma_\Gamma$ is discrete.  It follows that
  $H_\Sigma$ is finitary, being a coproduct of functors
  $\Posz(\Gamma, -)$ (each $\Posz(\Gamma, -)$ is finitary because
  $\Gamma$ is finite). 
  It follows that the free $H_{\Sigma}$-algebra over~$X$ is the
  colimit of the $\omega$-chain $(W_n)$ from~\eqref{eq:chain} in
  $\Pos$, where $W_0= X$ and $W_{n+1}=H_{\Sigma}W_n + X$ with connecting
  maps~$w_n$ as described ~\cite{Ada74}.\smnote{The citation should be
    deleted here because it indicates where this fact is proved, not
    where the chain is defined, which is above} The desired result
  thus follows from the concrete isomorphism
  $\AlgSigma\cong \Alg H_{\Sigma}$.
\end{proof}
\noindent A similar result can be proved for coherent
$\Sigma$-algebras and the associated functor $K_\Sigma$, using the
fact that like $\Posz(\Gamma,-)$, also the internal hom-functor
$\Pos(\Gamma,-)$ is finitary:
\begin{proposition}\label{P:freec}
  Every poset $X$ generates a free coherent $\Sigma$-algebra $\TcSigma X$. Its
  underlying poset is the union of the following $\omega$-chain of
  embeddings in $\Pos$:
  \[
    W_0= X
    \xra{w_0}
    W_1 = K_{\Sigma}X + X
    \xra{w_1}
    W_2 = K_{\Sigma}W_1 + X 
    \xra{w_3}
    \cdots
  \]
  The universal morphism $\eta_X\colon X \to \TcSigma X$ is the
  inclusion of~$W_0$ into the union.
\end{proposition}
\takeout{
\begin{definition}
  A \emph{term} in context $\Gamma$ is an element of the poset
  $\TSigma(\Gamma)$. We denote the ordering on terms by $\sqsubseteq$.
\end{definition}
\begin{remark}\label{E:term}
  Explicitly, for a poset~$X$, terms in $\TSigma X$ and their ordering
  $\sqsubseteq$ are generated inductively by the following rules:
  \begin{itemize}
  \item Every variable $x\in X$ is a term.
  \item If $x\le y$ for variables $x,y\in X$, then $x\sqsubseteq y$
    for the corresponding terms.
  \item If $\sigma\in\Sigma$ has arity~$\Gamma$ and
    $f\colon \Gamma\to \TSigma X$ is monotone, then $\sigma(f)$ is a
    term.
  \item Any term $\sigma(f)$ as per the previous item is comparable to
    itself.
  \end{itemize}
  Of course, the operation $\sigma_{\TSigma X}$ of $\TSigma X$ for
  $\sigma\in\Sigma_{\Gamma}$ assigns to
  $f\colon \Gamma\rightarrow\TSigma X$ the value $\sigma(f)$.

  The above description implies that there are rather fewer terms in
  $\TSigma X$ than one would maybe expect: By the above clauses, we
  have $t\sqsubseteq s$ for terms $t,s\in\TSigma X$ iff either $t=s$
  or $t=x,s=y$ for variables $x,y\in X$ such that $x\le
  y$. Consequently, whenever an operation $\sigma\in\Sigma_\Gamma$ has
  nondiscrete arity, say $x\le y$ for distinct $x,y\in\Gamma$, then
  terms of the form $\sigma(f)$ with~$\sigma,f$ as above exist in
  $\TSigma X$ only if either $f(x)=f(y)$, or both $f(x)$ and $f(y)$
  are variables in~$X$, and $f(x)\le f(y)$ in~$X$. For instance,
  for~$\Sigma$ being the signature of \autoref{E:lin}, all terms in
  context $\Gamma$ have have one of the forms
  \begin{equation*}
    x,\qquad x+y,\qquad @t,\qquad \text{or}\qquad t+t,
  \end{equation*}
  where $t$ is a term and $x,y$ are variables from $\Gamma$ such that
  $x\leq y$.

  In $\TcSigma X$, both the ordering and, consequently, the set of
  terms are larger than in $\TSigma X$, as we have the following
  additional rule for $\sqsubseteq$:
  \begin{itemize}
  \item Given $\sigma\in\Sigma_\Gamma$ and
    $f,g\colon\Gamma\to T_\Sigma(X)$, if $f\sqsubseteq g$ (pointwise)
    then $\sigma(f)\sqsubseteq\sigma(g)$.
  \end{itemize}
  E.g.\ in the signature of \autoref{E:lin}, given $x,y\in X$ such
  that $x\le y$, $\TcSigma X$ contains the term
  \begin{equation*}
    @x+@y
  \end{equation*}
  since $@x\sqsubseteq @y$ by the above rule; this term is not
  contained in $\TSigma X$.

\end{remark}

\begin{notation}\label{N:free}
  Let $A$ be a $\Sigma$-algebra. For every monotone function
  $f\colon \Gamma\rightarrow A$ (valuation of variables of $\Gamma$ in
  $A$) we denote by
  \[
    f^{\#}\colon \TSigma(\Gamma)\rightarrow A
  \]
  the corresponding homomorphism (interpretation of terms). For
  example, given $\sigma\in\Sigma_{\Gamma}$ we have
  $f^{\#}(\sigma(\eta_\Gamma))= \sigma_A(f)$.
\end{notation}
}
%
%
\begin{definition}\label{def:terms}
  We define \emph{terms} as usual in universal algebra, ignoring the
  order structure of arities; we write $\TermG$ for the set of
  $\Sigma$-terms in variables from $\Gamma$. Explicitly, the set
  $\TermG$ of terms is the least set containing~$|\Gamma|$ such
  that given an operation~$\sigma$ with arity $\Delta$ and a function
  $f\colon|\Delta|\to\TermG$, we obtain a term
  $\sigma(f)\in\TermG$.
\end{definition}

We denote by $u_\Gamma\colon \Gamma \to \TermG$ the inclusion
map.\smnote{This should \emph{not} be called $\eta$ because this is
  used as the unit of the monad $T$, which will be very confusing for
  the reader later.} We will often silently assume that the elements
of $|\Delta|$ are listed in some fixed sequence $x_1,\dots,x_n$, and
then write $\sigma(t_1, \ldots, t_n)$ in lieu of $\sigma(f)$ where
$f(x_i)=t_i$ for $i=1,\dots,n$. In particular, in examples we will
normally use arities $\Delta$ with $|\Delta|=\{1,\dots,k\}$ for
some~$k$, and then assume the elements of $\Delta$ to be listed in the
sequence $1,\dots,k$. We will often abbreviate $(t_1, \ldots, t_n)$ as
$(t_i)$, in particular writing $\sigma(t_i)$ in lieu of
$\sigma(t_1,\dots,t_n)$. Every $\sigma \in\Sigma_\Gamma$ yields the
term $\sigma(u_\Gamma)\in\TermG$, which by abuse of notation we will
occasionally write as just~$\sigma$.
\begin{example}
  Let $\Sigma$ be a signature with a single operations symbol $\sigma$
  whose arity is a $2$-chain. Then $\TermG$ is the set of usual terms
  for a binary operation on the variables from $\Gamma$. Whereas
  $\TSigma\Gamma$ contains only those terms which are variables or have the
  form $\sigma(t,t)$ for terms $t$ or $\sigma(x,y)$ for $x \leq y$ in
  $\Gamma$. The order of $\TSigma\Gamma$ is such that the only
  comparable distinct terms are the variables. 
\end{example}
\begin{definition}\label{D:sharp}
  Let $A$ be a $\Sigma$-algebra. Given a context $\Gamma$ (of
  variables) and a monotone interpretation $f\colon \Gamma \to A$, the
  \emph{evaluation map} is the partial map
  \[
    f^\#\colon \TermG \to |A|
  \]
  defined recursively by
  \begin{enumerate}
  \item $f^\#(x) = f(x)$ for every $x \in |\Gamma|$, and
  \item $f^\#(\sigma(g))$ is defined for $\sigma \in \Sigma_\Delta$
    and $g\colon|\Delta|\to \TermG$ iff all $f^\#(t_i)$ are
    defined and $i \leq j$ in $\Delta$ implies
    $f^\#(g(i)) \leq f^\#(g(j))$ in $A$; then
    $f^\#(\sigma(g)) = \sigma_A(f^\#\cdot g)$.
  \end{enumerate}
\end{definition}
\begin{example}\label{E:term}
  \begin{enumerate}
  \item For the signature in \autoref{E:lin}, we have terms in
    $\Term\{x,y\}$ such as $@x$, $y + @y$, etc. Given a
    $\Sigma$-algebra $A$ and an interpretation $f\colon\{x,y\}\to A$
    (say, with~$\{x,y\}$ ordered discretely), we see that $@x$ is
    always interpreted as $f^\#(@x)=@_A(f(x))$, whereas $f^\#(y + @x)$
    is defined if and only if $f(y) \leq @_A(f(x))$, and then
    $f^\#(y + @x)=f(y) +_A @_A(f(x))$.

  \item\label{E:term:2}
    \takeout{
    Recall from \autoref{def:terms} that every
    operation symbol $\sigma \in \Sigma_\Gamma$ defines a term
    $\sigma(x_i)$, where $|\Gamma| = \{x_1, \ldots, x_n\}$. Given any
    interpretation $f\colon \Gamma \to A$, since $f$ is monotone,
    $f^\#(\sigma(x_i))$ is defined, and we have}
  Every operation symbol $\sigma \in \Sigma_\Gamma$ considered as a
  term (see \autoref{def:terms}) satisfies
    \[
      f^\#(\sigma) = \sigma_A(f(x_i)).
    \]
  \end{enumerate}
\end{example}
\begin{definition}\label{def:ineqs}
  An \emph{inequation in context} $\Gamma$ is a pair $(s,t)$ of terms
  in $\TermG$, written in the form
  \[
    \Gamma\vdash s\leq t.
  \]
  Furthermore, we denote by
  \[
    \Gamma\vdash s = t
  \]
  the conjunction of the inequations $\Gamma\vdash s\leq t$ and
  $\Gamma\vdash t\leq s$.

  A $\Sigma$-algebra \emph{satisfies} $\Gamma \vdash s \leq t$ if for
  every monotone function $f\colon \Gamma \to A$, both $f^\#(s)$ and
  $f^\#(t)$ are defined and  $f^\#(s) \leq f^\#(t)$. 
\end{definition}
\begin{example}\label{E:lin2}
  For the signature of \autoref{E:lin}, consider the singleton context
  $\{x\}$ and the inequation
  \begin{equation}\label{Eqn:2.1}
    \{x\}\vdash x\leq @x. 
  \end{equation}
  An algebra $A$ satisfies this inequation iff $a\leq @_A(a)$ holds
  for every $a\in A$. In such algebras, the interpretation of the term
  $x+@x$ is defined everywhere. As a slightly more advanced example,
  consider the inequality (in the same signature)
  \begin{equation*}
    \{x\le y\}\vdash x + @x \le x.
  \end{equation*}
  According reading of inequalities as per Definition~\ref{def:ineqs},
  this inequality implies that $x + @x$ is always defined, which
  amounts precisely to \eqref{Eqn:2.1}.
\end{example}

\begin{definition} 
  A \emph{variety of $\Sigma$-algebras} is a full subcategory of
  $\AlgSigma$ specified by a set~$\E$ of inequations in context. We
  denote it by $\Alg(\Sigma, \E)$.\smnote{Please do not
    delete; it is used all the time.}\lsnote{Calling a set of
    \emph{in}equations `$\mathcal E$' is maybe not very suggestive, or
    rather suggestive of the wrong thing. How about $\mathcal I$ like
    in Chase's thesis?} Analogously, a \emph{variety of coherent
    $\Sigma$-algebras} is a full subcategory of $\AlgcSigma$ specified
  by a set of inequations in context.
\end{definition}

\begin{example}\label{ex:var}
We present some varieties of algebras.
\begin{enumerate}
\item We have seen a variety $\V$ specified by \eqref{Eqn:2.1} in \autoref{E:lin2}.
  
\item The subvariety of all coherent algebras in $\V$ can be specified
  as follows. Consider the 
contexts $\Gamma_1$ and $\Gamma_2$ given by
\[
  \Gamma_1 = 
  \begin{tikzcd}[sep = 30,baseline = -2]
    y
    \arrow[no head]{d}
    \\
    x
  \end{tikzcd}
  \qquad\text{and}\qquad
  \Gamma_2 = 
  \begin{tikzcd}[column sep = 10, row sep = 10, baseline=(B.base)]
                                           & y' \arrow[rd, no head] &                       \\
|[alias=B]|x' \arrow[rd, no head] \arrow[ru, no head] &                        & y \arrow[ld, no head] \\
                                           & x                      &                      
\end{tikzcd}
\]
and the inequations
\begin{equation}\label{Eqn:2.3}
  \Gamma_1\vdash @x\leq @y
  \qquad\text{and}\qquad
  \Gamma_2\vdash x+y\leq x'+y'.
\end{equation}
It is clear that $\Sigma$-algebras satisfying \eqref{Eqn:2.1} and \eqref{Eqn:2.3} form precisely the 
full subcategory of $\V$ consisting of coherent algebras.

\item\label{ex:var:3}
In general, all coherent $\Sigma$-algebras form a variety of $\Sigma$-algebras. 
For every context $\Gamma$, form the context $\bar{\Gamma}$ with variables $x$ 
and $x'$ for every variable $x$ of $\Gamma$, where the order is the least one such that 
the functions $e, e'\colon\Gamma\rightarrow\bar{\Gamma}$ given by $e(x)=x$ 
and $e'(x)=x'$ are embeddings such that $e\leq e'$. For every $\Gamma$ and 
every $\sigma\in\Sigma_{\Gamma}$ consider the following inequation in context 
$\bar{\Gamma}$:
\[\bar{\Gamma}\vdash \sigma(e)\leq\sigma(e').\]
It is satisfied by precisely those $\Sigma$-algebras $A$ for which
$\sigma_A$ is monotone.

\item\label{ex:var:4} 
  Recall that an \emph{internal semilattice} in a category with finite products
  is an object $A$ together with morphisms $+\colon A\times A\to A$ and
  $0\colon 1\to A$ such that
  \begin{enumerate}
  \item $0$ is a unit for $+$, i.e.~the following triangles
    commute\smnote{The notation was wrong here; given
      $f\colon X \to A$ and $g\colon X \to B$ the unique induced
      morphism is denoted by
      $\langle f, g\rangle\colon X \to A \times B$.}
    \[
      \begin{tikzcd}
        A \cong 1 \times A
        \arrow{r}{0 \times \id}
        \ar[equals]{rd}
        &
        A\times A
        \arrow{d}{+}
        &
        A \times 1 \cong A
        \ar{l}[swap]{\id \times 0}
        \ar[equals]{ld}
        \\
        &
        A
      \end{tikzcd}
    \]
    
  \item $+$ is associative, commutative, and idempotent:
    \[
      \begin{tikzcd}
        A \times A \times A
        \arrow{r}{+ \times \id}
        \arrow{d}[swap]{\id \times +}
        &
        A \times A
        \ar{d}{+}
        \\
        A\times A
        \ar{r}{+}
        &
        A
      \end{tikzcd}
      \qquad
      \begin{tikzcd}
        A \times A
        \ar{r}{\mathsf{swap}}
        \ar{rd}[swap]{+}
        & 
        A \times A
        \ar{d}{+}
        \\
        & A
      \end{tikzcd}
      \qquad
      \begin{tikzcd}
        A
        \ar{r}{\Delta}
        \ar[equals]{rd}
        &
        A \times A
        \ar{d}{+}
        \\
        &
        A
      \end{tikzcd}
    \]
    Here $\mathsf{swap} = \fpair{\pi_r, \pi_\ell}\colon A \times A \to A \times
    A$ is the canonical isomorphism commuting product components, and
    $\Delta = \fpair{\id,\id}\colon A \to A \times A$ is the diagonal.
  \end{enumerate}
  \noindent
  Internal semilattices in $\Pos$ form a
  variety of coherent $\Sigma$-algebras. To see this, consider the
  signature $\Sigma$ with $\Sigma_2 = \{+\}$ and
  $\Sigma_\mathbbm{\emptyset} = \{0\}$, where $2$ denotes the
  two-element discrete poset. The set $\E$ is formed by (in)equations
  specifying that $+$ is monotone, associative, commutative, and
  idempotent with unit $0$. Note that this does \emph{not} imply that
  $x + y$ is the join of $x, y$ in $X$ w.r.t.~its given order
  (cf.~\autoref{ex:freealgs}).

\item\label{ex:var:5} A related variety is that of classical
  join-semilattices (with $0$). To specify those, we take the signature $\Sigma$
  from the previous item; but now we need just two inequations
  in context specifying that $0$ and $+$ are the least element and the
  join operation, respectively:
  \[
    \{x\} \vdash 0 \leq x \qquad \{ x \leq z, y \leq z\} \vdash x+y \leq z.
  \]
  It then follows that $+$ is monotone, associative, commutative and
  idempotent, whence these equations need not be contained in $\E$. 
\item\label{item:bounded} \emph{Bounded joins:} Take the signature
  $\Sigma$ consisting of a unary operation~$\bot$ and an operation $j$
  (\emph{bounded join}) of arity $\{0,1,2\}$ where $0\le 2$ and
  $1\le 2$ (but $0\not\le 1$). We then define a variety~$\V$ by
  inequations in context
  \begin{gather*}
    x,y\vdash \bot(x)\le y\\
    x\le z, y\le z\vdash x\le j(x,y,z)\\ 
    x\le z, y\le z\vdash y\le j(x,y,z)\\
    x\le z, y\le z, x\le w, y\le w\vdash y\le j(x,y,z)\le w.
  \end{gather*}
  That is, $j(x,y,z)$ is the join of elements~$x,y$ having a joint
  upper bound~$z$. It follows that the value of $j(x,y,z)$, when it is
  defined, does not actually depend on~$z$, which instead just serves
  as a witness for boundedness of $\{x,y\}$. The operation~$\bot$ and
  its inequality specify that algebras are either empty or have a
  least element, i.e.\ the empty set has a join provided that it is
  bounded. Thus, $\V$ consists of the partial orders having all
  bounded finite joins, which we will refer to as \emph{bounded-join
    semilattices}, and morphisms in~$\V$ are monotone maps that
  preserve all existing finite joins.

\item\label{ex:var:7} Let a collection of posets $\Sigma_\Gamma$
  ($\Gamma \in \Posf$) be given. We obtain the corresponding signature
  $\Sigma^d = (|\Sigma_\Gamma|)_{\Gamma \in \Posf}$ by disregarding
  the order of $\Sigma_\Gamma$. Now consider the following set $\E$ of
  inequations in context:
  \[
    \Gamma \vdash \sigma(x_i) \leq \tau(x_i)
  \]
  where $|\Gamma| = \{x_1, \ldots, x_n\}$ and $\sigma, \tau \in
  \Sigma_\Gamma$ fulfil $\sigma \leq \tau$. Then the variety
  $\Alg(\Sigma, \E)$ is precisely the category of algebras for the
  non-discrete signature $\Sigma$ (see~\autoref{D:ndsig}). 
\end{enumerate}
\end{example}
\begin{remark}\label{R:create}
  We will now discuss limits and directed colimits in $\AlgSigma$.
  \begin{enumerate}
  \item It is easy to see that for every endofunctor $H$ on $\Pos$ the
    category $\Alg H$ of algebras for~$H$ is complete. Indeed, the
    forgetful functor $V\colon \Alg H\rightarrow\Pos$ creates
    limits. This means that for every diagram
    $D\colon\D\rightarrow \Alg H$ with $VD$ having a limit cone
    $(\ell_d\colon L\rightarrow VDd)_{d\in\text{obj}(\D)}$, there
    exists a unique algebra structure $\alpha\colon HL\rightarrow L$
    making each $\ell_d$ a homomorphism in $\Alg H$. Moreover, the
    cone $(\ell_d)$ is a limit of $D$.

  \item Analogously, it is easy to see that for every finitary
    endofunctor $H$ of $\Pos$ the category $\Alg H$ has filtered
    colimits created by $V$.

  \item We conclude from $\AlgSigma\cong \Alg H_\Sigma$ that limits
    and filtered colimits of $\Sigma$-algebras exist and are created
    by the forgetful functor into $\Pos$, and similarly for $\AlgcSigma$.

  \item\label{R:create:4} Moreover, we note that $\Alg H_\Sigma$ is a
    locally finitely presentable category; this was shown by
    Bird~\cite[Prop.~2.14]{Bird84}, see also the remark given by the
    first author and Rosick\'y~\cite[2.78]{AdamekR}.
  \end{enumerate}
\end{remark}
\begin{lemma}\label{L:compint}
  Let $A$ and $B$ be $\Sigma$-algebras, let $h\colon A \to B$ be a
  homomorphism, and let $f\colon\Gamma\to A$ be a monotone
  interpretation. Then for every term $t \in \TermG$
  we have that
  \begin{enumerate}
  \item\label{L:compint:1} $f^\#(t)$ is defined, $(h\cdot f)^\#(t)$ is also defined, and
    $(h\cdot f)^\#(t) = h(f^\#(t))$.
  \item if $h(f^\#(t))$ is defined and $h$ is an embedding, then
    $f^\#(t)$ is defined, too.
  \end{enumerate}
\end{lemma}
\begin{proof}
  \begin{enumerate}
  \item We proceed by induction on the structure of $t$. If~$t$ is a
    variable, then the claim is immediate from the definition of
    $(-)^\#$. For the inductive step, let~$t \in \TermG$ be a
    term of the form $t=\sigma(t_1,\dots, t_n)$ such that $f^\#(t)$ defined,
    where $\sigma\in\Sigma_{\Delta}$ and $|\Delta| = n$.  Then, by
    definition of $(-)^\#$, it follows that $f^\#(t_i)$ is defined for
    all $i\leq n$ and $f^\#(t_i)\leq f^\#(t_j)$ for all $i\leq j$ in
    $\Delta$ (i.e.~the map $i\mapsto f^\#(t_i)$ is monotone).
    Combining this with our assumption that $h\colon A\to B$ is a
    homomorphism, we obtain that
    \[
      h\cdot f^{\#}(\sigma(t_1,\dots, t_n))= \sigma_B(h\cdot f^{\#}(t_1),\dots, h\cdot f^\#(t_n)).
    \]
    Moreover, since $f^\#(t_i)$ is defined for all $i\leq n,$ the inductive hypothesis implies that
    $h\cdot f^{\#}(t_i) = (h\cdot f)^\#(t_i)$ for all $i\leq n$, hence also 
    \[
      (h\cdot f)^\#(t_i)= h\cdot f^{\#}(t_i)\leq h\cdot f^\#(t_j)= (h\cdot f)^\#(t_j)
    \]
    for all $i\leq j$ in $\Delta.$ Thus
    $\sigma_B((h\cdot f)^\#(t_1),\dots, (h\cdot f)^\#(t_n))$ is
    defined and equal to $h\cdot f^{\#}(\sigma(t_1,\dots, t_n)),$ as
    desired.
  \item Suppose now that $h$ is an embedding. We use a similar inductive
    proof.
    In the inductive step
    suppose that $(h \cdot f)^\#(t)$ is
    defined. Then by the definition of $(-)^\#$, it follows that
    $(h\cdot f)^\#(t_i)$ is defined for all $i \leq n$ and
    $(h \cdot f)^\#(t_i) \leq (h \cdot f)^\#(t_j)$ holds for all
    $i \leq j$ in~$\Delta$. By induction we know that all $f^\#(t_i)$
    are defined and by item~\ref{L:compint:1} that
    \[
      h \cdot f^\#(t_i)
      =
      (h \cdot f)^\#(t_i)
      \leq
      (h \cdot f)^\#(t_j)
      =
      h \cdot f^\#(t_i)
    \]
    holds for all $i \leq j$ in $\Delta$. Since $h$ is a embedding,
    we therefore obtain $f^\#(t_i) \leq f^\#(t_j)$ for all $i \leq j$
    in $\Delta$, whence $f^\#(t)$ defined.
    \qedhere
  \end{enumerate}
\end{proof}

\begin{proposition}\label{P:colim}
  Every variety is closed under filtered colimits in $\AlgSigma$.
\end{proposition}
\noindent
In other words, the full embedding $E\colon\V\subto\AlgSigma$ creates filtered
colimits.\smnote{We should mention this categorical statement!}
\begin{proof}
  Let $\V$ be a variety of $\Sigma$-algebras. Let
  $D\colon\D\rightarrow\AlgSigma$ be a filtered diagram having colimit
  $c_d\colon Dd\rightarrow A$ $(d\in\obj \D)$. It suffices to show
  that every inequation in context $\Gamma \vdash s \leq t$ satisfied
  by every algebra $Dd$ is also satisfied by $A$. Let
  $f\colon\Gamma\rightarrow A$ be a monotone interpretation. Since
  $\Gamma$ is finite, $f$ factorizes, for some $d\in\obj \D$, through
  $c_d$ via a monotone map $\bar{f}\colon \Gamma \to Dd$: in symbols,
  $c_d \cdot \bar f = f$. Since $Dd$ satisfies the given inequation in
  context, we know that $\bar f^\#(s)$ and $\bar f^\#(t)$ are defined
  and that $\bar f^\#(s) \leq \bar f^\#(t)$ in $Dd$. By
  \autoref{L:compint} we conclude that
  \[
    f^\#(s) = (c_d \cdot \bar f)^\#(s) = c_d \cdot \bar f^\#(s)
    \qquad
    \text{and}
    \qquad
    f^\#(t) = (c_d \cdot \bar f)^\#(t) = c_d \cdot \bar f^\#(t)
  \]
  are defined. Using the monotonicity of $c_d$ we obtain
  \[
    f^\#(s) = c_d \cdot \bar f^\#(s) \leq c_d \cdot \bar f^\#(t) =
    f^\#(t)
  \]
  as desired.
  \takeout{
  \begin{enumerate}
  \item\label{P:colim:1} We first prove the following property of
    filtered colimits. Let $D\colon \D \to \V$ be a filtered diagram
    with a colimit cocone $c_d\colon Dd \to A$ in $\AlgSigma$, where
    $d$ ranges over the objects of~$\D$. Given an object~$d$ in~$\D$
    and a monotone interpretation $f\colon \Gamma \to Dd$, for every
    term $s \in \TermG$ such that $(c_d \cdot f)^\#(s)$ is defined it
    follows that $(Dh \cdot f)^\#(s)$ is also defined for some
    morphism $h\colon d \to d'$ of $\D$.

    We prove this fact by structural induction. If $s$ is a variable
    in $\Gamma$, then put $h = \id_d$. For the inductive step, suppose
    that $s = \sigma(t_i)$ for $\sigma \in \Sigma_\Delta$ and
    $t_i \in \TermG$, for $1 \leq i \leq \card{\Delta}$. By
    definition of $(-)^\#$, definedness of
    $(c_d\cdot f)^\#(\sigma(t_i))$ implies that
    $(c_d \cdot f)^\#(t_i)$ is defined for all~$i$, and $i \leq j$ in
    $\Delta$ implies
    $(c_d \cdot f)^\#(t_i) \leq (c_d \cdot f)^\#(t_j)$. By
    \autoref{L:compint}, the latter is equivalent to
    $c_d (f^\#(t_i)) \leq c_d (f^\#(t_j))$. Since
    $\Delta \times \Delta$ is finite and $D$ is a filtered diagram, it
    follows that for some morphism $h$ of $\D$ we have
    \[
      (Dh \cdot f)^\#(t_i)
      =
      Dh(f^\#(t_i))
      \leq 
      Dh(f^\#(t_j)
      =
      (Dh \cdot f)^\#(t_j),
    \]
    whenever $i\le j$ in~$\Delta$, where the two equations follow by
    another application of \autoref{L:compint}. We therefore conclude
    that $(Dh \cdot f)^\#(s)$ is defined, as desired.

  \item We now prove that~$\V$ is closed under filtered colimits in
    $\AlgSigma$, as claimed. So let $D\colon\D\rightarrow\AlgSigma$ be
    a filtered diagram having colimit $c_d\colon Dd\rightarrow A$
    $(d\in\obj \D)$. It suffices to show that every inequation in
    context $\Gamma\vdash s\leq t$ satisfied by every algebra $Dd$ is
    also satisfied by $A$. Let $f\colon\Gamma\rightarrow A$ be a
    monotone interpretation. Since $\Gamma$ is finite, $f$ factorizes,
    for some $d\in\obj \D$, through $c_d$ via a monotone map
    $\bar{f}$. Furthermore, since $f^\#(s)$ and $f^\#(t)$ are defined,
    it follows from item~\ref{P:colim:1} that there exists a morphism
    $h\colon d \to d'$ in $\D$ such that $(Dh \cdot \bar f)^\#(s)$ and
    $(Dh \cdot \bar f)^\#(t)$ are defined:
    \[
      \begin{tikzcd}
        &
        Dd
        \arrow[d, "c_d"]
        \ar{r}{Dh}
        &
        Dd'
        \ar{ld}{c_{d'}}
        \\
        \Gamma \arrow[ru, "\bar{f}"]
        \arrow[r, "f"]
        &
        A                  
      \end{tikzcd}
    \]
    The given inequation is satisfied by $Dd'$, hence using
    \autoref{L:compint} we obtain
    \[
      Dh(\bar f^\#(s))
      = 
      (Dh \cdot \bar f)^{\#}(s)
      \leq
      (Dh \cdot \bar f)^{\#}(t)
      =
      Dh(\bar f^{\#}(t)).
    \]
    Since $c_{d'}$ is monotone and $c_d = c_{d'} \cdot Dh$, we
    conclude, again using \autoref{L:compint}, that
    \begin{align*}
      f^{\#}(s)
      & =
      (c_d \cdot \bar f)^\#(s)
      =
      c_d(\bar f^\#(s))
      =
      c_{d'} (Dh(\bar f^\#(s)))
      \\
      & \leq
      c_{d'} (Dh(\bar f^\#(t)))
      =
      c_d(\bar f^\#(t))
      =
      (c_d \cdot \bar f)^\#(t)
      =
      f^\#(t),
    \end{align*}
    as desired. \qedhere
  \end{enumerate}}
\end{proof}
\begin{corollary}\smnote{Maybe not needed but still interesting. I'd
    keep it as it cost us three lines. Our readers may want to have a
    source where this can be quoted.}
  The forgetful functor of a variety into $\Pos$ creates filtered
  colimits.
\end{corollary}
\noindent
Indeed, the forgetful functor of a variety $\V$ is a composite of the
inclusion $\V \subto \AlgSigma$ and the  forgetful functor of
$\AlgSigma$, which both create filtered colimits.
\begin{proposition}\label{P:refl}
  Every variety of $\Sigma$-algebras is a reflective subcategory of
  $\AlgSigma$ closed under subalgebras.
\end{proposition}
\begin{proof}
  \takeout{
  We are going to verify below that the factorization system (epi,
  embedding) of \autoref{R:fact} lifts from $\Pos$ to
  $\AlgSigma$. Then $\AlgSigma$ is complete (by \autoref{R:create})
  and cowellpowered. By~\cite[Thm.~16.8]{AHS90} every subcategory
  closed under $\mathcal{M}$-subalgebras (where $\mathcal{M}$ are the
  homomorphisms carried by embeddings) is reflective. Let
  $V\colon\AlgSigma\rightarrow\Pos$ denote the forgetful functor.
  \begin{enumerate}
  \item The factorization system (epi, embedding) lifts to
    $\AlgSigma$. Indeed, given a homomorphism
    $h\colon A\rightarrow B$, factorize $Vh$ as an epimorphism
    $e\colon V\!A\epito C$ followed by an embedding
    $m\colon C\monoto VB$ in $\Pos$. Then, for every $\Gamma$ and
    every $\sigma\in\Sigma_{\Gamma},$ there exists a unique operation
    $\sigma_C\colon\Posz(\Gamma, C)\rightarrow C$ for
    $\sigma\in\Sigma_{\Gamma}$ making $e$ and $m$ homomorphisms. Then
    the diagram below commutes:
    \[
      \begin{tikzcd}
        {\Posz(\Gamma, A)} \arrow[->>,d, "e\cdot (-)"'] \arrow[rr, "\sigma_A"]         &  & A \arrow[d, "e"] \\
        {\Posz(\Gamma, C)} \arrow[d, "m\cdot (-)"'] \arrow[rr, "\sigma_C", dashed] &  & C \arrow[>->,d, "m"] \\
        {\Posz(\Gamma, B)} \arrow[rr, "\sigma_B"]                             &  & B               
      \end{tikzcd}
    \]
    Indeed, this follows from $\Posz(\Gamma, e)= e \cdot (-)$ being an
    epimorphism.\smerror[inline]{This is false; $e\cdot (-)$ is not
      surjective.} It is easy to verify that the diagonal lift of this
    the diagram above provides the desired factorization.}

  We are going to prove below that every variety $\V = \Alg(\Sigma,\E)$ is closed in
  $\AlgSigma$ under products and subalgebras, whence it is closed
  under all limits. We also know from \autoref{P:colim} that $\V$ is
  closed under filtered colimits in $\AlgSigma$. Being a full
  subcategory of the locally finitely presentable category $\AlgSigma$
  (\autoref{R:create}\ref{R:create:4}), $\V$ is reflective by the
  reflection theorem for locally presentable
  categories~\cite[Cor.~2.48]{AdamekR}.
  
  \begin{enumerate}
  \item $\Alg(\Sigma, \E)$ is closed under products in
    $\Alg\Sigma$. Indeed, given $A=\prod_{i\in I}A_i$ with projections
    $\pi_i\colon A \to A_i$ and a monotone interpretation
    $f\colon \Gamma \to A$, we prove for every term
    $s \in \TermG$ that $f^\#(s)$ is defined if and only if so is
    $(\pi_i\cdot f)^\#(s)$ for all $i \in I$. This is done by
    structural induction: for $s \in |\Gamma|$ there is nothing to
    prove. Suppose that $s = \sigma(t_j)$ for some
    $\sigma \in \Sigma_\Delta$ and $t_j \in \TermG$,
    $j \in\Delta$. Then $f^\#(s)$ is defined iff
    $j \leq k$ in $\Delta$ implies $f^\#(t_j) \leq f^\#(t_k)$ in
    $A$. Equivalently (since the~$\pi_i$ are monotone and jointly
    order-reflecting, i.e.\ for every $x, y \in A$ we have $x \leq y$
    iff $\pi_i(x) \leq \pi_i(y)$ for all $i \in I$), $j \leq k$ in
    $\Delta$ implies $\pi_i\cdot f^\#(t_j) \leq \pi_i \cdot f^\#(t_k)$
    in $A_i$ for all $i \in I$. Since every $\pi_i$ is a homomorphism,
    this is equivalent to
    $(\pi_i \cdot f)^\#(t_j) \leq (\pi_i\cdot f)^\#(t_k)$ by
    \autoref{L:compint}.

    We now prove that $A$ satisfies every inequation
    $\Gamma \vdash s \leq t$ in $\E$, as claimed. Let
    $f\colon \Gamma \to A$ be a monotone interpretation. We have that
    $(\pi_i \cdot f^\#)(s)$ and $(\pi_i \cdot f^\#)(t)$ are defined
    and $\pi_i \cdot f^\#(s) \leq \pi_i\cdot f^\#(t)$ for all $i \in I$,
    using \autoref{L:compint} and since all $A_i$ satisfy the given
    inequation in context. Using again that the~$\pi_i$ are jointly
    order-reflecting, we obtain $f^{\#}(s)\leq f^{\#}(t)$, as required.
    \takeout{
    for every inequation $ \Gamma\vdash s\leq t$ in $\E$ we
    prove that $A$ satisfies this inequation. For a valuation
    $f\colon \Gamma\rightarrow A$ we know that the interpretations
    $\pi_i\cdot f\colon\Gamma\rightarrow A_i$ are such that
    $ (\pi_i\cdot f)^{\#}(s)\leq(\pi_i\cdot f)^{\#}(t)$ for every
    $i\in I$. Since $\pi_i$ is a homomorphism, we have
    $(\pi_i\cdot f)^{\#} =\pi_i\cdot f^{\#}$. From
    $\pi_i\cdot f^{\#}(s)\leq\pi_i\cdot f^{\#}(t)$, for every
    $i\in I$, we conclude $f^{\#}(s)\leq f^{\#}(t)$ since $\pi_i$ is
    monotone. Thus $A$ lies in $\Alg(\Sigma, \E)$.}%

  \item\label{P:refl:3} $\Alg(\Sigma, \E)$ is closed under subalgebras
    in $\AlgSigma$.  Indeed, let $m\colon B\subto A$ be a
    $\Sigma$-homomorphism carried by an embedding. For every
    inequation $\Gamma\vdash s\leq t$ in $\E$ we prove that $B$
    satisfies it. For a monotone interpretation
    $f\colon \Gamma\rightarrow B$, we see that
    $(m \cdot f)^\#(s)$ and $(m \cdot f)^\#(t)$ are defined and $(m
    \cdot f)^\#(s) \leq (m \cdot f)^\#(t)$ since $A$ satisfies the
    given inequation in context. By \autoref{L:compint} we obtain that
    $f^\#(s)$ and $f^\#(t)$ are defined and
    \[
      m \cdot f^\#(s)
      =
      (m \cdot f)^\#(s)
      \leq
      (m \cdot f)^\#(t)
      =
      m \cdot f^\#(s).
    \]
    Since $m$ is an embedding, it follows that
    $f^{\#}(s)\leq f^{\#}(t)$.\qedhere
  \end{enumerate}
\end{proof}
\begin{corollary}
  The category $\AlgcSigma$ of all coherent $\Sigma$-algebras is a
  reflective subcategory of $\AlgSigma$.
\end{corollary}
\noindent
Indeed, this follows using \autoref{ex:var}\ref{ex:var:3}.

\begin{theorem}\label{T:varmon}
For every variety, the forgetful functor to $\Pos$ is monadic.
\end{theorem}
\begin{proof}
  Let $\V$ be a variety of $\Sigma$-algebras. We use Beck's Monadicity 
  Theorem~\cite[Thm.~VI.7.1]{MacLane98}  and prove that the forgetful 
  functor $U\colon\V\to\Pos$ has a left adjoint and creates coequalizers of 
  $U$-split pairs.
  \begin{enumerate}
  \item The functor $U$ has a left adjoint because it is the composite
    of the embedding $E\colon\V\rightarrow\AlgSigma$ and
    the forgetful functor $V\colon\AlgSigma\rightarrow\Pos$: the
    functor $E$ has a left adjoint by \autoref{P:refl} and $V$ has one
    by \autoref{P:free}.


  \item Let $f,g\colon A\rightarrow B$ be a $U$-split pair of
    homomorphisms in $\V$. That is, there are monotone
    maps $c,i,j$ as in the following diagram
    \[
      \begin{tikzcd}
        UA \arrow[r, "Uf", shift left] \arrow[r, "Ug"', shift right] & UB \arrow[r, "c"] \arrow[l, "j", bend left=60] & C \arrow[l, "i", bend 		left=60]
      \end{tikzcd}
    \]
    satisfying
    $c\cdot Uf=c\cdot Ug$, $c\cdot i=\id_C$, $Uf\cdot j=\id_{UB}$, and $Ug\cdot j=i\cdot c$.
    
    For every $\sigma\in\Sigma_{\Gamma}$, there exists a unique
    operation $\sigma_C\colon\Posz(\Gamma, C)\rightarrow C$\cfnote{Changed to Posz from Pos.} making $c$
    a homomorphism:
    \[
      \begin{tikzcd}
        {\Posz(\Gamma, B)}
        \arrow[r, "\sigma_B"]
        \arrow[d, "c\cdot(-)"']
        &
        B \arrow[d, "c"]
        \\
        {\Posz(\Gamma, C)} \arrow[r, "\sigma_C"']
        &
        C 
      \end{tikzcd}
    \]
    Indeed, let us define $\sigma_C$ by 
    \[
      \sigma_C(h)=c\cdot \sigma_B(i\cdot h)\qquad\text{for all $h\colon\Gamma\rightarrow C$}.
    \]
    Then $c$ is a homomorphism since
    $\sigma_C(c\cdot k)=c\cdot \sigma_B(k)$ for every
    $k\colon\Gamma\rightarrow B$:
    \begin{align*}
      c\cdot\sigma_B(k)
      &=c\cdot \sigma_B(f\cdot j\cdot k) &
      \text{since $f\cdot j=\id$} \\
      &=c\cdot f\cdot \sigma_A(j\cdot k) & \text{$f$ a homomorphism} \\
      &=c\cdot g\cdot\sigma_A(j\cdot k) & \text{since $c\cdot f=c\cdot g$} \\ 
      &=c\cdot \sigma_B(g\cdot j\cdot k) & \text{$g$ a homomorphism} \\
      &=c\cdot \sigma_B(i\cdot c\cdot k) & \text{since $g\cdot j=i\cdot c$} \\
      &=\sigma_C(c\cdot k).
    \end{align*}
    Conversely, if $C$ has an algebra structure making $c$ a
    homomorphism, then the above formula holds since $c\cdot i= \id$:
    \[
      \sigma_C(h)=\sigma_C(c\cdot i\cdot h)=c\cdot \sigma_B(i \cdot h).
    \]
    
    Furthermore, $C$ lies in $\V$. To verify this, we just 
    prove that whenever an inequation $\Gamma\vdash s\leq t$ is satisfied by 
    $B$, then the same holds for the algebra $C$. Given a monotone
    interpretation $h\colon\Gamma\rightarrow C$ such that $h^\#(s)$ and
    $h^\#(t)$ are defined, we prove $h^{\#}(s)\leq h^{\#}(t)$.
    
    For the monotone interpretation $i\cdot h\colon\Gamma\rightarrow B$ we have that
    $(i\cdot h)^\#(s)$ and $(i \cdot h)^\#(t)$ are defined and that
    $(i\cdot h)^{\#}(s)\leq(i\cdot h)^{\#}(t)$ since $B$ lies in $\V$. Since $c$ is a homomorphism, we
    conclude using \autoref{L:compint} and that $c \cdot i = \id_C$ that
    \[
      h^\#(s) = (c \cdot i \cdot h)^\# (s) = c \cdot (i\cdot h)^\# (s)
    \]
    is defined and similarly for $h^\#(t)$. Then we have 
    \[
      h^{\#}(s)= c\cdot (i\cdot h)^{\#}(s)\leq c\cdot
      (i\cdot h)^{\#}(t) = h^{\#}(t).
    \]
    as desired using the monotonicty of $c$.

    Finally, we prove that $c$ is a coequalizer of $f$ and $g$ in $\V$. Let
    $d\colon B \to D$ be a homomorphism such that $d\cdot f = d \cdot
    g$. Then $d' = d \cdot i$ fulfils $d = d'\cdot c$:
    \begin{align*}
      d' \cdot c &= d \cdot i \cdot c\\
      &= d\cdot g \cdot j & \text{since $i\cdot c = g \cdot j$} \\
      &= d \cdot f \cdot j & \text{since $d \cdot f = d \cdot g$} \\
      &= d & \text{since $f \cdot j = \id_B$.}
    \end{align*}
    Moreover, $d'\colon C \to D$ is a homomorphism since $c$ is a
    surjective homomorphism such that $d'\cdot c = d$ is also a
    homomorphism. This clearly is the unique homomorphic factorization
    of $d$ through $c$. 
    \qedhere
  \end{enumerate}
\end{proof}
\begin{definition}
  Given a variety $\V$, the left adjoint of $U\colon\V\to\Pos$
  assigns to every poset $X$ the free algebra of $\V$ on $X$. The
  ensuing monad is called the \emph{free-algebra monad} of the variety
  and is denoted by $\T_\V$.
\end{definition}
\begin{corollary}\label{C:varmon}
  Every variety $\V$ is
  isomorphic, as a concrete category over $\Pos$, to the
  Eilenberg-Moore category $\Pos^{\T_\V}$.
\end{corollary}
\begin{example}\label{ex:freealgs}
  \begin{enumerate}
  \item Recall the variety of internal semilattices considered in
    \autoref{ex:var}\ref{ex:var:4}. It is well known (and easy to
    show) that the free internal semilattice on a poset $X$ is formed
    by the poset $\Conv X$ of its finitely generated convex
    subsets. Here, a subset $S \subseteq X$ is \emph{convex} if
    $x, y \in S$ implies that every $z$ such that $x \leq z \leq y$
    lies in $S$, too, and \emph{finitely generated} means that $S$ is
    the convex hull of a finite subset of $X$. The order on
    $C_\omega X$ is the Egli-Milner order, which means that for
    $S,T \in \Conv X$ we have
    \[
      S \leq T \quad\text{iff}\quad
      \forall s \in S.\,\exists t \in B.\, s \leq t \wedge
      \forall t \in T.\,\exists s \in S.\, s \leq t.
    \]
    The constant $0$ is the empty set, and the operation $+$ is the
    join w.r.t.~inclusion, explicity, $S + T$ is the convex hull of
    $S \cup T$ for all $S, T \in \Conv X$. One readily shows that $+$
    is monotone w.r.t.~the Egli-Milner order and that $\Conv X$ with
    the universal monotone map $x \mapsto \{x\}$ is a free internal
    semilattice on $X$. Thus we see that $\Conv$ is a monad on $\Pos$
    and $\Pos^\Conv$ is (isomorphic to) the category of internal
    semilattices in $\Pos$.
    
  \item
  Denote by $D_{\omega}$ the monad of free join semilattices. It assigns to
  every poset $X$ the set of finitely generated, downwards closed subsets of 
  $X$ ordered by inclusion. Here a downwards closed subset $S\subseteq X$ is
  \emph{finitely generated} if there are $x_1, \ldots, x_n \in S$, $n \in \N$, such that
  $S = \bigcup_{i = 1}^n x_i\mathord{\downarrow}$. The category $\Pos^{D_{\omega}}$
  is equivalent to that of join-semilattices, see \autoref{ex:var}\ref{ex:var:5}.
\item Similarly, the monad $D^b_\omega$ generated by the variety of
  bounded-join semilattices (\autoref{ex:var}\ref{item:bounded})
  assigns to a poset~$X$ the set of finitely generated downwards
  closed \emph{bounded} subsets of~$X$, ordered by inclusion.
  \end{enumerate}
\end{example}
\begin{corollary}
  The forgetful functors $U\colon \AlgSigma \to \Pos$ and $U_c\colon
  \AlgcSigma\to \Pos$ are monadic. 
\end{corollary}
\noindent
Note that the corresponding monads are the
free-(coherent-)$\Sigma$-algebra monads given by $\TSigma X$ and
$\TcSigma X$, respectively (cf.~\autoref{P:free} and~\ref{P:freec}). 

\section{Finitary Monads}

Let $\T$ be a finitary monad on $\Pos$. We present a variety $\V_{\T}$
such that the mapping $\T\mapsto\V_{\T}$ is inverse to the assignment
$\V \to \T_\V$ of a variety to its free-algebra monad. Moreover, we
prove that there is a completely analogous bijection between enriched
finitary monads and varieties of coherent algebras.

\begin{remark}\label{R:KT}
Let us recall the equivalence between the category of monads on $\Pos$ and Kleisli triples
established by Manes~\cite[Thm 3.18]{Manes76}. 
\begin{enumerate}
\item\label{R:KT:1} A \emph{Kleisli triple} consists of (a)~a self map
  $X\mapsto TX$ on the class of all posets, (b)~an assignment of a
  monotone map $\eta_X\colon X\to TX$ to every poset, and (c)~an
  assignment of a monotone map $f^*\colon TX\to TY$ to every monotone
  map $f\colon X\to TY$, which satisfies
  \begin{align}
    \eta^*_X &= \id_{X^{*}} \label{KT1} \\
    f^*\cdot \eta_X &= f \label{KT2} \\
    g^*\cdot f^* &= (g^* \cdot f)^* \label{KT3}
  \end{align}
  for all posets $X$ and all monotone functions $f\colon X\to TY$ and
  $g\colon Y\to TZ$.
\item\label{R:KT:2} A morphism into another Kleisli triple
  $(T', \eta', (-)^+)$ is a collection $\varphi_X\colon TX\to T'X$ of
  monotone functions such that 
  the diagrams below commute for all posets $X$
  and all monotone functions $f\colon X\to TY$:
\[
  \begin{tikzcd}
    &
    X \arrow[rd, "\eta_X'"]
    \arrow[ld, "\eta_X"'] & &  &
    TX \arrow[r, "\varphi_X"] \arrow[d, "f^*"']
    &
    T'X \arrow[d, "(\varphi_Y\cdot f)^+"]
    \\
    TX \arrow[rr, "\varphi_X"]
    &
    & T'X &  & TY \arrow[r, "\varphi_Y"]                   & T'Y 
\end{tikzcd}
\]
\item\label{R:KT:3}
  Every monad $\T$ defines a Kleisli triple $(T, \eta, (-)^*)$ by 
  \[
    f^*= TX\xra{Tf} TTY\xra{\mu_Y} TY.
  \]
  Every monad morphism $\varphi\colon\T\to\T'$ defines a morphism
  $\varphi_X\colon TX\to T'X$ of Kleisli triples. The resulting
  functor from the category of monads to the category of Kleisli
  triples is an equivalence functor.
\end{enumerate}
\end{remark}

\noindent We shall now define the variety $\V_{\T}$ mentioned
above.\smnote{Numbered remark is necessary since it is refered to (as
  one can see from the fact that it carries a label! LS: That label is
  not used anywhere, so I removed the remark again}
%

%
\begin{definition}\label{D:var}
  The \emph{variety $\V_{\T}$ associated} to a finitary monad $\T$ on
  $\Pos$ has the signature
  \[
    \Sigma_{\Gamma}= |T\Gamma|\qquad\text{for every $\Gamma\in\Posf$}.
  \]
  That is, operations of arity $\Gamma$ are elements of the poset
  $T\Gamma$.  For each $\Gamma$, we impose inequations of the
  following two types:
  \begin{enumerate}
  \item\label{D:var:1} $\Gamma\vdash \sigma\leq\tau$ for all
    $\sigma\leq\tau$ in $T\Gamma$ (with operations used as terms as
    per \autoref{def:terms}), and
  \item\label{D:var:2} $\Gamma \vdash k^*(\sigma) = \sigma(k)$ for all 
    $\Delta\in\Posf$, monotone $k\colon\Delta\rightarrow T\Gamma$ and
    $\sigma \in T\Delta$. 
    \takeout{%
    monotone we form the tuple $\hat k\colon |\Delta| \to
    \TermG$ of terms by putting $\hat k(x) = k(x)$, where the
    operation symbol $k(x)$ in $|T\Gamma|$ is considered as a term
    (cf.~\autoref{R:opterm}). For every $\sigma$ in $T\Delta$ we pose
    the equation
    \[
      \Gamma \vdash k^*(\sigma) = \sigma(\hat k)
    \]
    where $k^*:=\mu_{T\Gamma}\cdot Tk$.}
  \end{enumerate}
\end{definition}
\begin{example}\label{E:TX}
  For every poset $X$, the poset $TX$ carries the following structure
  of an algebra of $\vt$. Given $\sigma\in T\Gamma$, we define the
  operations $\sigma_{TX}\colon\Posz(\Gamma, TX)\rightarrow TX$ by
  \[
    \sigma_{TX}(f)=f^*(\sigma)\qquad\text{for $f\colon\Gamma\rightarrow TX$}.
  \]
  It then follows that the evaluation map
  $f^{\#}\colon\TermG\rightarrow |TX|$ coincides with $f^*$ on
  operation symbols (converted to terms as per \autoref{def:terms}):
  \begin{equation}\label{eq:fsigma}
    f^{\#}(\sigma)=f^*(\sigma)
  \end{equation}
  for all $\sigma\in T\Gamma$.
  \takeout{
    Indeed, $f^{\#}$ is a homomorphism\smnote[inline]{We do not know
      this anymore!}:
    \begin{equation}\label{eq:sharphom}
      \begin{tikzcd}
        {\Posz(\Gamma, \TSigma(\Gamma))}
        \arrow[d, "f^{\#}\cdot (-)"']
        \arrow[r, "\sigma_{\TSigma(\Gamma)}"]
        &
        \TSigma(\Gamma)
        \arrow[d, "f^{\#}"]
        \\
        {\Posz(\Gamma, TX)} \arrow[r, "\sigma_{TX}"']
        &
        TX
      \end{tikzcd}
    \end{equation}
    Recall that $\sigma$ is the term $\sigma(\eta)$ for
    $\eta\colon\Gamma\rightarrow\TSigma(\Gamma)$, thus applied to $\eta$
    the above square yields
    \[
      f^{\#}(\sigma)=\sigma_{TX}(f^{\#}\cdot\eta)=\sigma_{TX}(f)=f^*(\sigma).
    \]}
  Indeed, for $|\Gamma| =\{x_1, \ldots, x_n\}$ we have
  \begin{align*}
    f^\#(\sigma)
    &= f^\#(\sigma(x_1, \ldots, x_n))
    & \text{\autoref{def:terms}}
    \\
    &= \sigma_{TX} (f^\#(x_1), \ldots, f^\#(x_n))
    & \text{def.~of $f^\#$}
    \\
    & = \sigma_{TX} (f(x_1), \ldots, f(x_n))
    & \text{def.~of $f^\#$}
    \\
    &= \sigma_{TX}(f) \\
    &= f^*(\sigma)
    & \text{def.~of $\sigma_{TX}$.}
  \end{align*}
  It now follows that the $\Sigma$-algebra $TX$ lies in $\vt$. It
  satisfies the inequations of type~\ref{D:var:1} because $f^*$ is
  monotone: given $\sigma\leq\tau$ in $T\Gamma$, we have
  $f^{\#}(\sigma)=f^*(\sigma)\leq f^*(\tau)=f^{\#}(\tau)$. Moreover,
  it satisfies the inequations of type~\ref{D:var:2} since for every monotone
  map $k\colon\Delta\rightarrow T\Gamma$ we know that
  $f^\#(k^*(\sigma))$ is defined by \autoref{E:term}\ref{E:term:2}, and
  we have
  \begin{align*}
    f^{\#}(k^*(\sigma))
    &= f^*\cdot k^*(\sigma) & \text{by~\eqref{eq:fsigma}} \\
    &=(f^*\cdot k)^*(\sigma) & \text{by~\eqref{KT3}} \\
    &= \sigma_{TX}(f^*\cdot k) & \text{def.~of $\sigma_{TX}$} \\
    &= \sigma_{TX}(f^{\#}\cdot k) & \text{by~\eqref{eq:fsigma}} \\
    &=f^{\#}(\sigma(k)) & \text{def.~of $f^{\#}$}
  \end{align*}
  So, indeed, $TX$ lies in $\vt$.%
  
  \smnote{Should we explain which $\Sigma$ and $\E$ are induced
    by $\Conv$ and $\powd$? SM: Ok, let's leave it for the time being
    and perhaps do it with the revisions.}
\end{example}

\begin{theorem}\label{T:mon-var}
  Every finitary monad $\T$ on $\Pos$ is the free-algebra monad of its
  associated variety $\vt$.
\end{theorem}
\begin{proof}
  \begin{enumerate}
  \item\label{T:mon-var:1} We first prove that the algebra $TX$ of
    \autoref{E:TX} is a free algebra of $\vt$ w.r.t.~the monad unit
    $\eta_{X}\colon X\rightarrow TX$.
    \begin{enumerate}[label=(1\alph*)]
    \item First, suppose that $X=\Gamma$ is a context. Given an
      algebra $A$ of $\vt$ and a monotone map
      $f\colon\Gamma\rightarrow A$, we are to prove that there exists
      a unique homomorphism $\bar{f}\colon T\Gamma\rightarrow A$ such that
      $f=\bar{f}\cdot\eta$.

      Indeed, given $\sigma\in T\Gamma$, define $\bar{f}$ by
      \[
        \bar{f}(\sigma)=\sigma_A(f).
      \]
      This is a monotone function: if $\sigma\leq\tau$ in $T\Gamma$,
      then use the fact that $A$ satisfies the
      inequations~$\Gamma \vdash \sigma \leq \tau$ to obtain
      \[
        \sigma_A(f)=f^{\#}(\sigma)\leq f^{\#}(\tau)=\tau_A(f).
      \]
      We now verify that $\bar{f}$ is a homomorphism: given
      $\tau\in\Sigma_{\Delta}$, we will prove that the following
      square commutes:
      \[
        \begin{tikzcd}
          {\Posz(\Delta, T\Gamma)}
          \arrow[d, "\bar{f}\cdot (-)"']
          \arrow[r, "\tau_{T\Gamma}"]
          &
          T\Gamma \arrow[d, "\bar{f}"]
          \\
          {\Posz(\Delta, A)} \arrow[r, "\tau_A"']
          &
          A                           
        \end{tikzcd}
      \]
      Indeed, for every monotone map $k\colon\Delta\rightarrow
      T\Gamma$ we have that $f^\#$ is defined in
      $k^*(\tau)$ by \autoref{E:term}\ref{E:term:2}, and we therefore
      obtain (letting $|\Delta| = \{x_1, \ldots, x_n\}$):
      \begin{align*}
        \bar f(\tau_{T\Gamma}(k))
        &= \bar{f}(k^{*}(\tau))
        & \text{def.~of $\tau_{T\Gamma}$}
        \\
        &= (k^*(\tau))_A(f)
        &\text{def.~of $\bar f$}
        \\
        &= f^\#(k^*(\tau))
        & \text{by \autoref{D:sharp}}
        \\
        &= f^\#(\tau(\hat k))
        &
        \text{$A$ satisfies $\Gamma \vdash k^*(\tau) = \tau(\hat k)$}
        \\
        &= \tau_A(f^\#(k))
        & \text{def.~of $f^\#$}
        \\
        &= \tau_A(\bar f \cdot k).
      \end{align*}
      For the last
      step we use again the definition of $f^\#$ to obtain that for
      every $x \in |\Delta|$ the operation symbol $\sigma = k(x)$,
      considered as the term $\sigma(y_1, \ldots, y_k)$ where
      $|\Gamma| = \{y_1, \ldots, y_k\}$ (\autoref{def:terms}), satisfies
      \begin{align*}
        f^\# (\sigma(y_1, \ldots,y_k)) &= \sigma_A (f^\#(y_1),
        \ldots, f^\#(y_k)) = \sigma_A (f(y_1), \ldots, f(y_k)) \\
        & = \sigma_A (f) = \bar f(\sigma_i).
      \end{align*}
      Since $\sigma = k(x_i)$ this gives the desired
      $\bar f \cdot k$ when we let $x$ range over $\Delta$.
      \takeout{
      The left-hand side is $\tau_A(g)$ where
      $g\colon\Delta\rightarrow A$ is given by $g(x)=(k(x))_A(f)$. For
      the evaluation map $f^{\#}\colon\TermG\to |A|$,
      this is precisely $f^{\#}(\tau_{T\Gamma}(k))$ (see
      \autoref{E:term}\ref{E:term:2}).  For the right-hand side we
      have, following \autoref{D:sharp}:
      \[
        \bar{f}(k^*(\tau))=(k^*(\tau))_A(f)=f^{\#}(k^*(\tau)).
      \]
      This is the same result because $A$ satisfies
      $\Gamma\vdash k^*(\tau)=\tau(k)$ and $f^\#$ is defined in
      $k^*(\tau)$ by \autoref{E:term}\ref{E:term:2}.}

      As for uniqueness, suppose that
      $\bar{f}\colon T\Gamma\rightarrow A$ is a homomorphism such that
      $f=\bar{f}\cdot\eta_{\Gamma}$. The above square commutes for
      $\Delta=\Gamma$ which applied to
      $\eta_{\Gamma}\in\Pos(\Gamma, T\Gamma)$ yields for every $\sigma
      \in |T\Gamma|$:
      \begin{align*}
        \bar{f}(\sigma)
        &= \bar f(\eta_\Gamma^*(\sigma))
        & \text{by~\eqref{KT1}}
        \\
        &= \bar f (\eta_\Gamma^\#(\sigma))
        &\text{by~\eqref{eq:fsigma}}
        \\
        &= \bar{f}(\sigma_{T\Gamma}(\eta_{\Gamma}))
        & \text{def.~of $\eta_\Gamma^\#$}
        \\
        & =\sigma_A(\bar{f}\cdot\eta_{\Gamma})
        & \text{$\bar f$ homomorphism}
        \\
        &= \sigma_A(f)
        & \text{since $\bar f\cdot\eta_\Gamma = f$},
      \end{align*}
      as required.

    \item Now, let $X$ be an arbitrary poset. Express it as a filtered
      colimit $X=\colim_{i\in I}\Gamma_i$ of contexts. The free
      algebra on $X$ is then a filtered colimit of the corresponding
      diagram of the $\Sigma$-algebras $T\Gamma_i$ ($i\in I$). Indeed,
      that $TX=\colim T\Gamma_i$ in $\Pos$ follows from $T$ preserving
      filtered colimits. That this colimit lifts to $\V$ follows
      from the forgetful functor of $\V$ creating filtered colimits,
      see \autoref{P:colim}.
\end{enumerate}

\item To conclude the proof, we apply \autoref{R:KT}. Our given monad
  and the monad $\T_{\V}$ of the associated variety share the same
  object assignment $X\mapsto TX= T_{\V}X$ for an arbitrary poset $X$,
  and the same universal map $\eta_X$, as shown in
  part~\ref{T:mon-var:1}. It remains to prove that for every morphism
  $f\colon X\rightarrow TY$ in $\Pos$ the homomorphism
  $h^*=\mu_{Y}\cdot Th$ extending $h$ in $\Pos^{\T}$ is a
  $\Sigma$-homomorphism $h^*\colon TX\rightarrow TY$ of the
  corresponding $\Sigma$-algebras of \autoref{E:TX}. Then $\T$ and
  $\tv$ also share the operator $h\mapsto h^*$.  Thus given
  $\sigma\in\Sigma_{\Gamma}$ we are to prove that the following square
  commutes:
\[
\begin{tikzcd}
{\Posz(\Gamma, TX)} \arrow[d, "h^*\cdot (-)"'] \arrow[r, "\sigma_{TX}"] & TX \arrow[d, "h^*"] \\
{\Posz(\Gamma, TY)} \arrow[r, "\sigma_{TY}"']                      & TY                 
\end{tikzcd}
\]
Indeed, given $f\colon\Gamma\rightarrow TX$ we have 
\begin{align*}
  h^*\cdot \sigma_{TX}(f)
  &=h^*\cdot f^*(\sigma) & \text{definition of $\sigma_A$} \\
  &=(h^*\cdot f)^*(\sigma) & \text{equation~\eqref{KT3}} \\ 
  &=\sigma_{TY}(h^*\cdot f) & \text{definition of $\sigma_{TY}$}
\end{align*}
This completes the proof. \qedhere
\end{enumerate}
\end{proof}

\begin{corollary}\label{C:nonenriched}
  Finitary monads on $\Pos$ correspond bijectively, up to monad isomorphism, 
  to finitary varieties of ordered algebras.
\end{corollary}

\noindent Indeed, the assignment of the associated variety $\vt$ to
every finitary monad $\T$ is essentially inverse to the asignment of
the free-algebras monad $\T_{\V}$ to every variety $\V$. To see this,
recall that every variety $\V$ is isomorphic (as a concrete category
over $\Pos$) to the category $\Pos^{\T_{\V}}$
(\autoref{C:varmon}). Conversely, every finitary monad $\T$ is
isomorphic to $\T_{\V}$ for the associated variety
(\autoref{T:mon-var}).

\begin{proposition}\label{P:cohvar}
  If $\T$ is an enriched finitary monad on $\Pos$, then the algebras
  of its associated variety $\V_\T$ are coherent. Conversely, for
  every variety $\V$ of coherent algebras, the free-algebra monad
  $\T_{\V}$ is enriched.
\end{proposition}
\begin{proof}
  For the first claim, let $\T$ be enriched. 
  Then the $\Sigma$-algebra $TX$ of \autoref{E:TX} is
  coherent: Given an operation symbol $\sigma\in\Sigma_{\Gamma}$ and
  monotone interpretations $f\leq g$ in $\Pos(\Gamma, TX)$, we have
  $Tf\leq Tg$, and hence
  $f^*=\mu_{TX}\cdot Tf\leq\mu_{TX}\cdot Tg=g^*$ because~$\T$ is
  enriched. Therefore, $f^*(\sigma)\leq g^*(\sigma)$. That is,
  \[
    \sigma_{TX}(f)\leq\sigma_{TX}(g).
  \]
  For every algebra $A$ of the variety $\vt$ we have the unique
  $\Sigma$-homomorphism $k\colon TA\rightarrow A$ such that
  $k\cdot\eta_A=\id_A$ (since $TA$ is a free $\Sigma$-algebra
  in $\vt$; see \autoref{T:mon-var}\ref{T:mon-var:1}). The coherence
  of $TA$ implies the coherence of $A$: given $f_1\leq f_2$ in
  $\Pos(\Gamma, A)$, we verify $\sigma_A(f_1)\leq\sigma_A(f_2)$ by
  applying the commutative square 
  \[
    \begin{tikzcd}
      {\Pos(\Gamma, TA)} \arrow[r, "\sigma_{TA}"] \arrow[d,
      "k\cdot (-)"']
      &
      TA \arrow[d, "k"]
      \\
      {\Pos(\Gamma, A)} \arrow[r, "\sigma_A"']
      &
      A                        
    \end{tikzcd}
  \]
  to $\eta_A\cdot f_i$, obtaining
  $ \sigma_A(f_i) = \sigma_A(k\cdot \eta_A\cdot f_i) = k\cdot
  \sigma_{TA}(\eta_A\cdot f_i)$; by monotonicity of composition in
  $\Pos$ and of $\sigma_{TA}$ as established above, this implies
  $\sigma_A(f_1)\leq\sigma_A(f_2)$ as desired.

  Conversely, let $\V$ be a variety of coherent
  $\Sigma$-algebras. Given $f_1\leq f_2$ in $\Pos(X, Y)$, we prove
  that the free-algebra monad $\tv$ fulfils $T_{\V}f_1\leq
  T_{\V}f_2$. Let $e\colon E\hookrightarrow T_\V X$ be the subposet of
  all elements $t\in|T_{\V}X|$ such that
  $T_{\V}f_1(t)\leq T_{\V}f_2(t)$.  Since for $x\in X$ we know that
  $f_1(x)\leq f_2(x)$, the poset $E$ contains all elements
  $\eta_X(x)$. Moreover, $E$ is closed under the operations of
  $T_{\V}X$: Suppose that $\sigma\in\Sigma_{\Gamma}$ and that
  $h\colon\Gamma\rightarrow T_{\V}X$ is a monotone map such that
  $h[\Gamma]\subseteq E$; we have to show that
  $\sigma_{T_{\V}X}(h)\in E$.  Applying the  commutative
  square
  \[
    \begin{tikzcd}
      {\Pos(\Gamma, T_{\V}X)} \arrow[r, "\sigma_{T_{\V}X}"] \arrow[d,
      "T_{\V}f_i\cdot (-)"']
      &
      T_{\V}X \arrow[d, "T_{\V}f_i"]
      \\
      {\Pos(\Gamma, T_{\V}Y)} \arrow[r, "\sigma_{T_{\V}Y}"']
      &
      T_{\V}Y                                
    \end{tikzcd}
  \]
  to $h$, we obtain 
  \begin{align*}
    T_{\V}f_1(\sigma_{T_{\V}X}(h))
    &= \sigma_{T_{\V}Y}(T_{\V}f_1\cdot h) \\
    &\leq \sigma_{T_{\V}Y}(T_{\V}f_2\cdot h) \\
    &= T_{\V}f_2(\sigma_{T_{\V}X}(h))
  \end{align*}
  using in the inequality that $\sigma_{T_{\V}Y}$ is monotone and, by
  assumption, $T_{\V}f_1(h)\leq T_{\V}f_2(h)$; that is,
  $\sigma_{T_{\V}X}(h)\in E$, as desired.

  We thus see that $E$ is a $\Sigma$-subalgebra of $T_{\V}X$. Since
  $T_{\V}X$ is the free algebra of $\V$ w.r.t.~$\eta_X$ and the
  subalgebra $E$ contains $\eta_X[X]$, it follows that
  $E=T_{\V}X$. This proves that $Tf_1\leq Tf_2$, as desired.
\end{proof}

\begin{corollary}\label{C:enriched}
  Enriched finitary monads on $\Pos$ correspond bijectively, up to
  monad isomorphism, to finitary varieties of coherent ordered algebras.
\end{corollary}

\section{Enriched Lawvere Theories}\label{S:enriched}

Power~\cite{Pow99} proves that enriched finitary monads on $\Pos$
bijectively correspond to Lawvere $\Pos$-theories. This is another way
of proving \autoref{C:enriched}. However, we believe that a precise
verification of all details would not be simpler than our proof. Here
we indicate this alternative proof.

Dual to \autoref{R:tensor}, \emph{cotensors} $P \cotensor X$ in an
enriched category $\Tcat$ (over $\Pos$) are characterized by an
enriched natural isomorphism $\Tcat(-,P\cotensor X) \cong \Pos(P,
\Tcat(-,X))$. If we restrict ourselves to finite posets $P$ we speak
about \emph{finite cotensors}. 
\removeThmBraces
\begin{definition}[{\cite{Pow99}}]
  A \emph{Lawvere $\Pos$-theory} is a small enriched category
  $\Tcat$ with finite cotensors together with an enriched
  identity-on-objects functor $\iota\colon \Posf^{\op} \to \Tcat$
  which preserves finite cotensors.
\end{definition}
\resetCurThmBraces
\begin{example}\label{E:TV}
  Let $\V$ be a variety, and denote by $\T_\V$ its free-algebra monad
  on $\Pos$. The following theory $\Tcat_\V$ is the restriction of the
  Kleisli category of $\T_\V$ to $\Posf$: objects are all contexts,
  and morphisms from $\Gamma$ to $\Gamma'$ form the poset $\Pos(\Gamma',
  T_\V \Gamma)$. A composite of $f\colon \Gamma' \to T_\V\Gamma$ and
  $g\colon \Gamma'' \to T_\V\Gamma'$ is $f^*\cdot g\colon \Gamma'' \to
  T_\V\Gamma$ where $(-)^*$ is the Kleisli extension (see
  \autoref{R:KT}\ref{R:KT:3}). 
\end{example}
\removeThmBraces
\begin{theorem}[{\cite[Thm.~4.3]{Pow99}}]\label{T:Pow}
    There is a bijective correspondence between enriched finitary
    monads on $\Pos$ and Lawvere $\Pos$-theories.
\end{theorem}
\resetCurThmBraces
\begin{example}\label{E:Powproof}
  By inspecting Power's proof, we see that for the theory $\Tcat_\V$
  of \autoref{E:TV}, the corresponding monad is precisely the
  free-algebra monad $\T_\V$.
\end{example}
\begin{remark}
  With every Lawvere $\Pos$-theory $\Tcat$, Power associates the category
  $\Mod \Tcat$ of \emph{models}, which are enriched functors $\bar A\colon \Tcat \to
  \Pos$ preserving finite cotensors. Morphisms are all enriched
  natural transformations between models.

  In \autoref{E:TV}, every algebra $A$ of $\V$ yields a model $\bar A$
  of $\Tcat_\V$ by putting $\bar A(\Gamma) = \V(T_\V\Gamma, A)$ and
  for $f\colon \Gamma' \to T_\V\Gamma$ we have
  \[
    \bar A(f) = f^*\cdot (-)\colon \V(T_\V\Gamma, A) \to \V(T_\V\Gamma', A).
  \]
  The proof of \autoref{T:Pow} implies that these are, up to
  isomorphism, all models of $\Tcat_\V$ and this yields an equivalence
  between $\V$ and $\Mod\Tcat_\V$.
\end{remark}
Thus, \autoref{C:enriched} can be proved by verifying that every
Lawvere $\Pos$-theory $\Tcat$ is naturally isomorphic to $\Tcat_\V$
for a variety of algebras, and the passage from $\T$ to $\V$ is
inverse to the passage $\V \mapsto \Tcat_\V$ of \autoref{E:Powproof}.

\smnote{The following should not be in a numbered remark to be skipped
  by readers, but narrative.}  In addition, Nishizawa and
Power~\cite{NP09} generalize the concept of Lawvere theory to a
setting in which one may obtain an alternative proof of the
non-coherent case (\autoref{C:nonenriched}); we briefly indicate
how. Again we believe that that proof would not be simpler than
ours. The setting of op.\ cit.\ includes a symmetric monoidal closed
category $\V$ that is locally finitely presentable in the enriched
sense and a locally finitely presentable $\V$-category $\A$. For our
purposes, $\V = \Set$ and $\A = \Pos$.
\removeThmBraces
\begin{definition}[{\cite[Def.~2.1]{NP09}}]\label{D:Lawth}
  A \emph{Lawvere $\Pos$-theory} for $\V = \Set$ is a small ordinary
  category $\Tcat$ together with an ordinary identity-on-objects functor $\iota\colon
  \Posf^\op \to \Tcat$ preserving finite limits.
\end{definition}
\resetCurThmBraces
\begin{example}
  Every variety of (not necessarily coherent) algebras yields a theory
  $\Tcat$ analogous to \autoref{E:TV}: the hom-set
  $\Tcat(\Gamma,\Gamma')$ is $\Posz(\Gamma', T_\V \Gamma)$.
\end{example}
\begin{remark}
  Here, a model of a theory $\Tcat$ is an ordinary functor $A\colon
  \Tcat \to \Set$ such that $A \cdot \iota\colon \Posf^\op \to \Set$
  is naturally isomorphic to $\Pos(-,X)/\Posf^\op$ for some poset
  $X$. The category $\Mod \Tcat$ of models has ordinary natural
  transformations as morphisms.
\end{remark}
\removeThmBraces
\begin{theorem}[{\cite[Cor.~5.2]{NP09}}]
  There is a bijective correspondence between ordinary finitary monads
  on $\Pos$ an Lawvere $\Pos$-theories in the sense of \autoref{D:Lawth}.
\end{theorem}
\resetCurThmBraces

\section{Conclusion and Future Work}

Classical varieties of algebras are well known to correspond to
finitary monads on $\Set$. We have investigated the analogous
situation for the category of posets. It turns out that there are two
reasonable variants: one considers either all (ordinary) finitary
monads, or just the enriched ones, whose underlying endofunctor is
locally monotone.  (An orthogonal restriction, not considered here, is
to require the monad to be strongly finitary, which corresponds to
requiring the arities of operations to be discrete~\cite{ADV20}.) We
have defined the concept of a variety of ordered algebras using
signatures where arities of operation symbols are finite posets. We
have proved that these varieties bijectively correspond to
\begin{enumerate}
\item all finitary monads on Pos, provided that algebras are not
  required to have monotone operations, and
\item all enriched finitary monads on $\Pos$ for varieties of coherent
  algbras, i.e.~those with monotone operations. 
\end{enumerate}
\noindent
In both cases, `term'
has the usual meaning in universal algebra, and varieties are
classes presented by inequations in context.

Although we have concentrated entirely on posets, many
features of our paper can clearly be generalized to enriched locally
$\lambda$-presentable categories and the question of a semantic
presentation of (ordinary or enriched) $\lambda$-accessible monads.
For example, what type of varieties corresponds to countably
accessible monads on the category of metric spaces with distances at
most one (and nonexpanding maps)? Such varieties will be related to
Mardare et al.'s quantitative varieties~\cite{MardareEA16}
(aka.~$c$-varieties~\cite{MardareEA17,MU19}), probably extended by
allowing non-discrete arities of operation symbols.

Ji\v{r}\'i Rosick\'y (private communication) has suggested another possibility of presenting
finitary monads on $\Pos$: by applying the functorial semantics of
Linton~\cite{Linton69} to functors into $\Pos$ and taking the
appropriate finitary variation in the case where those functors are
finitary. We intend to pursue this idea in future work.

%
%
\smnote[inline]{\cite{ADV20} should be put on arXiv; we cannot cite
  something that is not available to referees.}

\bibliography{monadsrefs}

\providecommand{\noopsort}[1]{}
\begin{thebibliography}{10}

\bibitem{Ada74}
J.~Ad\'{a}mek.
\newblock Free algebras and automata realizations in the language of
  categories.
\newblock {\em Comment.\ Math.\ Univ.\ Carolin.}, pp. 589--602, 1974.

\bibitem{ADV20}
J.~Ad\'amek, M.~Dost\'al, and J.~Velebil.
\newblock A categorical view of varieties of ordered algebras.
\newblock {S}ubmitted, available at \url{https://arxiv.org/abs/2011.13839},
  2020.

\bibitem{AdamekR}
J.~Ad\'{a}mek and J.~Rosick\'y.
\newblock {\em Locally Presentable and Accessible Categories}.
\newblock Cambridge University Press, 1994.

\bibitem{Bar70}
M.~Barr.
\newblock Coequalizers and free triples.
\newblock {\em Math.\ Z.}, 116(4):307--322, 1970.

\bibitem{Bird84}
R.~Bird.
\newblock {\em Limits in 2-categories of locally presentened categories}.
\newblock PhD thesis, University of Sidney, 1984.

\bibitem{Bloom76}
S.~Bloom.
\newblock Varieties of ordered algebras.
\newblock {\em J.\ Comput.\ System Sci.}, pp. 200--212, 1976.

\bibitem{BW83}
S.~Bloom and J.~Wright.
\newblock P-varieties -- a signature independent characterization of varieties
  of ordered algebras.
\newblock {\em J.\ Pure Appl.\ Algebra}, pp. 13--58, 1983.

\bibitem{Borceux94-2}
F.~Borceux.
\newblock {\em Handbook of Categorical Algebra: Volume 2, Categories and
  Structures}.
\newblock Encyclopedia of Mathematics and its Applications. Cambridge
  University Press, 1994.

\bibitem{Kelly80}
{\relax G.M}.~Kelly.
\newblock A unified treatment of transfinite constructions for free algebras,
  free monoids, colimits, associated sheaves, and so on.
\newblock {\em Bull.~Austral.~Math.~Soc.}, 22:1--83, 1980.

\bibitem{KL93}
{\relax G.M}.~Kelly and S.~Lack.
\newblock Finite product-preserving functors, kan extensions, and
  stronlgy-finitary 2-monads.
\newblock {\em Appl.\ Categ.\ Structures}, 1(1):85--94, 1993.

\bibitem{KP93}
{\relax G.M}.~Kelly and {\relax A.J}.~Power.
\newblock Adjunctions whose counits are coequalizers, and presentations of
  finitary enriched monads.
\newblock {\em J.\ Pure Appl.\ Algebra}, pp. 163--179, 1993.

\bibitem{KV17}
A.~Kurz and J.~Velebil.
\newblock Quasivarieties and varieties of ordered algebras: regularity and
  exactness.
\newblock {\em Math.\ Structures Comput.\ Sci.}, pp. 1153--1194, 2017.

\bibitem{Lac99}
S.~Lack.
\newblock On the monadicity of finitary monads.
\newblock {\em J.\ Pure Appl.\ Algebra}, 140(1):65--73, 1999.

\bibitem{Linton69}
F.~E. Linton.
\newblock An outline of functorial semantics.
\newblock In B.~Eckmann, ed., {\em Seminar on Triples and Categorical Homology
  Theory}, vol.~80 of {\em Lecture Notes Math.}, pp. 7--52. Springer, 1969.

\bibitem{MacLane98}
S.~MacLane.
\newblock {\em Categories for the Working Mathematician}.
\newblock Springer, 2nd edition, 1998.

\bibitem{Manes76}
E.~Manes.
\newblock {\em Algebraic Theories}.
\newblock Springer, 1976.

\bibitem{MardareEA16}
R.~Mardare, P.~Panangaden, and G.~Plotkin.
\newblock Quantitative algebraic reasoning.
\newblock In M.~Grohe, E.~Koskinen, and N.~Shankar, eds., {\em Logic in
  Computer Science, {LICS} 2016}, pp. 700--709. {ACM}, 2016.

\bibitem{MardareEA17}
R.~Mardare, P.~Panangaden, and G.~Plotkin.
\newblock On the axiomatizability of quantitative algebras.
\newblock In {\em 32nd Annual {ACM/IEEE} Symposium on Logic in Computer
  Science, {LICS} 2017, Reykjavik, Iceland, June 20-23, 2017}, pp. 1--12.
  {IEEE} Computer Society, 2017.

\bibitem{MU19}
S.~Milius and H.~Urbat.
\newblock Equational axiomatization of algebras with structure.
\newblock In M.~Boja\'nczyk and A.~Simpson, eds., {\em Proc.~22nd International
  Conference on Foundations of Software Science and Computation Structures
  (FoSSaCS 2019)}, vol. 11425 of {\em Lecture Notes Comput.~Sci.}, pp.
  400--417. Springer, 2019.

\bibitem{NP09}
K.~Nishizawa and {\relax A.J}.~Power.
\newblock Lawvere theories enriched over a general base.
\newblock {\em J.\ Pure Appl.\ Algebra}, 213(3):377--386, 2009.

\bibitem{PP01}
G.~Plotkin and {\relax A.J}.~Power.
\newblock Semantics for algebraic operations.
\newblock {\em Electron.\ Notes in Theor.\ Comput.\ Sci.}, 45:332--345, 2001.
\newblock Seventeenth Conference on the Mathematical Foundations of Programming
  Semantics, Proc.\ MFPS 2001.

\bibitem{PP02}
G.~Plotkin and {\relax A.J}.~Power.
\newblock Notions of computation determine monads.
\newblock In {\em Foundations of Software Science and Computation Structures,
  5th International Conference, Proc.\ {FoSSaCS} 2002}, vol. 2303 of {\em LNCS
  2002}, pp. 342--356. Springer Verlag, 2002.

\bibitem{Pow99}
{\relax A.J}.~Power.
\newblock Enriched lawvere theories.
\newblock {\em Theory Appl.\ Categories}, pp. 83--93, 1999.

\bibitem{TrnkovaEA75}
V.~Trnkov\'{a}, J.~Ad\'{a}mek, V.~Koubek, and V.~Reiterman.
\newblock Free algebras, input processes and free monads.
\newblock {\em Comment. Math. Univ. Carolin.}, 16:339--351, 1975.

\end{thebibliography}
\bibliographystyle{myabbrv}

\end{document}